\documentclass[distribution,final]{brandiss} % Pretty version for distribution (*not* for GSAS submission!)

\usepackage{argthesis}

\numberwithin{section}{chapter}
\numberwithin{figure}{chapter}

\disstitle{$\Gamma$-species, Quotients, and Graph Enumeration}
\dissauthor{Andrew Gainer}
\dissadvisor{Ira Gessel}
\dissdepartment{Mathematics}
\dissmonth{May}
\dissyear{2012}
\dissdean{Malcom Watson}

\begin{document}

% GSAS formatting requirements: Front page order is title, signature,
% copyright, dedication, acknowledgements, abstact, preface, table of
% contents, list of figures, list of tables.

% REQUIRED: The start of roman-numbered plain pages.
\frontmatter

% REQUIRED: Create the dissertation title page.
\makedisstitle

% REQUIRED: Create the signature page.  Add one line for each member
% of your dissertation committee, except for your advisor who is
% automatically added before the rest.
\begin{disssignatures}
  \committeemember Ji Li, Dept.\ of Mathematics
  \committeemember Thomas Roby, Dept.\ of Mathematics, University of Connecticut
\end{disssignatures}

\disscopyright % optional

\begin{dissdedication}
  For my friends, who keep me sane.
\end{dissdedication}

\begin{dissacknowledgments} % recommended
  Without Ira's brilliant and understanding mentorship, this research would not have been possible.

  Without Susan's friendship and guidance or the camaraderie of the members of my cohort, I never would have made it to the point of doing graduate research.

  Without the encouragement David, Margaret, and Kim gave me, I might never have discovered mathematical research at all.

  Without my parents' patient and generous support or the forbearance of my many teachers from childhood into my graduate years, I would never even have discovered the academy.

  Without Janet's affection, fellowship, and tolerance of my eccentricities, my progress in these last years---and my spirits---would have been greatly diminished. 

  My life and this work are gift from and a testament to everyone whom I have known.
  Thank you all.
\end{dissacknowledgments}

% REQUIRED: The dissertation abstract.
\begin{dissabstract}
  The theory of $\Gamma$-species is developed to allow species-theoretic study of quotient structures in a categorically rigorous fashion.
  This new approach is then applied to two graph-enumeration problems which were previously unsolved in the unlabeled case---bipartite blocks and general $k$-trees.
\end{dissabstract}

\begin{disspreface}
  Historically, the algebra of generating functions has been a valuable tool in enumerative combinatorics.
  The theory of combinatorial species uses category theory to justify and systematize this practice, making clear the connections between structural manipulations of some objects of interest and algebraic manipulations of their associated generating functions.
  The notion of `quotient' enumeration (that is, of counting orbits under some group action) has been applied in species-theoretic contexts, but methods for doing so have largely been ad-hoc.
  We will contribute a species-compatible way for keeping track of the way a group $\Gamma$ acts on structures of a species $F$, yielding what we term a $\Gamma$-species, which has the sort of synergy of algebraic and structural data that we expect from species.
  We will then show that it is possible to extract information about the $\Gamma$-orbits of such a $\Gamma$-species and harness this new method to attack several unsolved problems in graph enumeration---in particular, the isomorphism classes of nonseparable bipartite graphs and $k$-trees (that is, `unlabeled' bipartite blocks and $k$-trees).

  It is assumed that the reader of this thesis is familiar with the classical theory of groups and that he has encountered at least the basic vocabularies of category theory and graph theory.
  Results in these fields which are not original to this thesis will either be referenced from the literature or simply assumed, depending on the degree to which they are part of the standard body of knowledge one acquires when studying those disciplines.

  In the first chapter, we outline the theory of species, develop several classical methods, and introduce the notion of a $\Gamma$-species.
  In the second chapter, we apply these techniques to the enumeration of unlabeled vertex-$2$-connected bipartite graphs, a historically open problem.
  In the third chapter, we apply these techniques to the more complex problem of the enumeration of unlabeled general $k$-trees, also historically unsolved.
  Finally, in an appendix we discuss algebraic and computational methods which allow species-theoretical insights to be translated into explicit algorithmic techniques for enumeration.
\end{disspreface}

\setcounter{tocdepth}{2}
\tableofcontents % REQUIRED

\listoffigures % optional

\listoftables % optional

% REQUIRED: The start of arabic-numbered fancy pages.
\mainmatter

\chapter{The theory of species}\label{c:species}
\section{Introduction}\label{s:introspec}
Many of the most important historical problems in enumerative combinatorics have concerned the difficulty of passing from `labeled' to `unlabeled' structures.
In many cases, the algebra of generating functions has proved a powerful tool in analyzing such problems.
However, the general theory of the association between natural operations on classes of such structures and the algebra of their generating functions has been largely ad-hoc.
Andr\'{e} Joyal's introduction of the theory of combinatorial species in \cite{joy:species} provided the groundwork to formalize and understand this connection.
A full, pedagogical exposition of the theory of species is available in \cite{bll:species}, so we here present only an outline, largely tracking that text.

To begin, we wish to formalize the notion of a `construction' of a structure of some given class from a set of `labels', such as the construction of a graph from its vertex set or or that of a linear order from its elements.
The language of category theory will allow us capture this behavior succinctly yet with full generality:
\begin{definition}\label{def:species}
  Let $\catname{FinBij}$ be the category of finite sets with bijections and $\catname{FinSet}$ be the category of finite sets with set maps.
  Then a \emph{species} is a functor $F: \catname{FinBij} \to \catname{FinSet}$.
  For a set $A$ and a species $F$, an element of $F \sbrac{A}$ is an \emph{$F$-structure on $A$}.
  Moreover, for a species $F$ and a bijection $\phi: A \to B$, the bijection $F \sbrac{\phi}: F \sbrac{A} \to F \sbrac{B}$ is the \emph{$F$-transport of $\phi$}.
\end{definition}
A species functor $F$ simply associates to each set $A$ another set $F \sbrac{A}$ of its $F$-structures; for example, for $\specname{S}$ the species of permutations, we associate to some set $A$ the set $\specname{S} \sbrac{A} = \operatorname{Bij} \pbrac{A}$ of self-bijections (that is, permutations as maps) of $A$.
This association of label set $A$ to the set $F \sbrac{A}$ of all $F$-structures over $A$ is fundamental throughout combinatorics, and functorality is simply the requirement that we may carry maps on the label set through the construction.

\begin{example}
  \label{ex:graphspecies}
  Let $\specname{G}$ denote the species of simple graphs labeled at vertices.
  Then, for any finite set $A$ of labels, $G \sbrac{A}$ is the set of simple graphs with $\abs{A}$ vertices labeled by the elements of $A$.
  For example, for label set $A = \sbrac{3} = \cbrac{1, 2, 3}$, there are eight graphs in $\specname{G} \sbrac{A}$, since there are $\binom{3}{2} = 3$ possible edges and thus $2^{3} = 8$ ways to choose a subset of those edges:
  \begin{equation*}
    \specname{G} \sbrac{\cbrac{1, 2, 3}} = \cbrac{
      \begin{array}{c}
        \begin{aligned}
          \tikz{ \node[style=graphnode](1) at (90:1) {1};
            \node[style=graphnode](2) at (210:1) {2};
            \node[style=graphnode](3) at (330:1) {3};
          }        
        \end{aligned},
        \begin{aligned}
          \tikz{ \node[style=graphnode](1) at (90:1) {1};
            \node[style=graphnode](2) at (210:1) {2};
            \node[style=graphnode](3) at (330:1) {3};
            \draw(1) to (2);
          }
        \end{aligned},
        \begin{aligned}
          \tikz{ \node[style=graphnode](1) at (90:1) {1};
            \node[style=graphnode](2) at (210:1) {2};
            \node[style=graphnode](3) at (330:1) {3};
            \draw(2) to (3);
          }
        \end{aligned},
        \begin{aligned}
          \tikz{ \node[style=graphnode](1) at (90:1) {1};
            \node[style=graphnode](2) at (210:1) {2};
            \node[style=graphnode](3) at (330:1) {3};
            \draw(1) to (3);
          }
        \end{aligned}, \\
        \begin{aligned}
          \tikz{ \node[style=graphnode](1) at (90:1) {1};
            \node[style=graphnode](2) at (210:1) {2};
            \node[style=graphnode](3) at (330:1) {3};
            \draw(1) to (2);
            \draw(1) to (3);
            \draw(2) to (3);
          }
        \end{aligned},
        \begin{aligned}
          \tikz{ \node[style=graphnode](1) at (90:1) {1};
            \node[style=graphnode](2) at (210:1) {2};
            \node[style=graphnode](3) at (330:1) {3};
            \draw(1) to (3);
            \draw(2) to (3);
          }
        \end{aligned},
        \begin{aligned}
          \tikz{ \node[style=graphnode](1) at (90:1) {1};
            \node[style=graphnode](2) at (210:1) {2};
            \node[style=graphnode](3) at (330:1) {3};
            \draw(1) to (2);
            \draw(1) to (3);
          }
        \end{aligned},
        \begin{aligned}
          \tikz{ \node[style=graphnode](1) at (90:1) {1};
            \node[style=graphnode](2) at (210:1) {2};
            \node[style=graphnode](3) at (330:1) {3};
            \draw(1) to (2);
            \draw(2) to (3);
          }
        \end{aligned}
      \end{array}
    }.
  \end{equation*}

  The symmetric group $\symgp{3}$ acts on the set $\sbrac{3}$ as permutations.
  Consider the permutation $\pmt{(23)}$ that interchanges $2$ and $3$ in $\sbrac{3}$.
  Then $\specname{G} \sbrac{\pmt{(23)}}$ is a permutation on the set $\specname{G} \sbrac{\cbrac{1, 2, 3}}$; for example,
  \begin{equation*}
    \specname{G} \sbrac{\pmt{(23)}} \pbrac{
        \begin{aligned}
          \tikz{ \node[style=graphnode](1) at (90:1) {1};
            \node[style=graphnode](2) at (210:1) {2};
            \node[style=graphnode](3) at (330:1) {3};
            \draw(1) to (2);
          }
        \end{aligned}
      } =
      \begin{aligned}
        \tikz{ \node[style=graphnode](1) at (90:1) {1};
          \node[style=graphnode](2) at (210:1) {2};
          \node[style=graphnode](3) at (330:1) {3};
          \draw(1) to (3);
        }.
      \end{aligned}.
    \end{equation*}
\end{example}

Since the image of a bijection under such a functor is necessarily itself a bijection, many authors instead simply define a species as a functor $F: \catname{FinBij} \to \catname{FinBij}$.
Our motivation for using this definition instead will become clear in \cref{s:quot}.

Note that, having defined the species $F$ to be a functor, we have the following properties:
\begin{itemize}
\item for any two bijections $\alpha: A \to B$ and $\beta: B \to C$, we have $F \sbrac{\alpha \circ \beta} = F \sbrac{\alpha} \circ F \sbrac{\beta}$, and
\item for any set $A$, we have $F \sbrac{\Id_{A}} = \Id_{F \sbrac{A}}$.
\end{itemize}
Accordingly, we (generally) need not concern ourselves with the details of the set $A$ of labels we consider, so we will often restrict our attention to a canonical label set $\sbrac{n} := \cbrac{1, 2, \dots, n}$ for each cardinality $n$.
Moreover, the permutation group $\symgp{A}$ on any given set $A$ acts by self-bijections of $A$ and induces \emph{automorphisms} of $F$-structures for a given species $F$.
The orbits of $F$-structures on $A$ under the induced action of $\symgp{A}$ are then exactly the `unlabeled' structures of the class $F$, such as unlabeled graphs.

Finally, we note that it is often natural to speak of maps between classes of combinatorial structures, and that these maps are sometimes combinatorially `natural'.
For example, we might wish to map the species of trees into the species of general graphs by embedding; to map the species of connected bicolored graphs to the species of connected bipartite graphs by forgetting some color information; or the species of graphs to the species of sets of connected graphs by identification.
These maps are all `natural' in the sense that they are explicitly structural and do not reference labels; thus, at least at a conceptual level, they are compatible with the motivating ideas of species.
We can formalize this notion in the language of categories:
\begin{definition}
  \label{def:specmap}
  Let $F$ and $G$ be species. A \emph{species map} $\phi$ is a natural transformation $\phi: F \to G$ --- that is, an association to each set $A \in \catname{FinBij}$ of a set map $\phi_{A} \in \catname{FinSet}$ such that the following diagram commutes:
  \begin{equation*}
    \begin{tikzpicture}[every node/.style={fill=white}]
      \matrix (m) [matrix of math nodes, row sep=4em, column sep=4em]
      {
        F \sbrac{A} & F \sbrac{B} \\
        G \sbrac{A}  & G \sbrac{B} \\
      };
      \path[->,font=\scriptsize]
      (m-1-1) edge node {$\phi_{A}$} (m-2-1)
      edge node {$F \sbrac{\sigma}$} (m-1-2)
      (m-2-1) edge node {$G \sbrac{\sigma}$} (m-2-2)
      (m-1-2) edge node {$\phi_{B}$} (m-2-2);
    \end{tikzpicture}
  \end{equation*}
  We call the set map $\phi_{A}$ the \emph{$A$ component of $\phi$} or the \emph{component of $\phi$ at $A$}.
\end{definition}

Such species maps may capture the idea that two species are essentially `the same' or that one `contains' or `sits inside' another.
\begin{definition}
  \label{def:specmaptypes}
  Let $F$ and $G$ be species and $\phi: F \to G$ a species map between them.
  In the case that the components $\phi_{A}$ are all bijections, we say that $\phi$ is a \emph{species isomorphism} and that $F$ and $G$ are \emph{isomorphic}.
  In the case that the components $\phi_{A}$ are all injections, we say that $\phi$ is a \emph{species embedding} and that $F$ \emph{embeds in} $G$ (denoted $\phi: F \hookrightarrow G$).
  Likewise, in the case that the components $\phi_{A}$ are all surjections, we say that $\phi$ is a \emph{species covering} and that $F$ \emph{covers} $G$ (denoted $\phi: F \twoheadrightarrow G$).
\end{definition}

With the full power of the language of categories, we may make the following more general observation:
\begin{note}
  \label{note:specmapcattheo}
  Let $\catname{Spc}$ denote the functor category of species; that is, define $\catname{Spc} \defeq \catname{FinSet}^{\catname{FinBij}}$, the category of functors from $\catname{FinBij}$ to $\catname{FinSet}$.
  Species maps as defined in \cref{def:specmap} are natural transformations of these functors and thus are exactly the morphisms of $\catname{Spc}$.
  
  It is a classical theorem of category theory (cf.\ \cite{mac:cftwm}) that the epi- and monomorphisms of a functor category are exactly those whose components are epi- and monomorphisms in the target category if the target category has pullbacks and pushouts.
  Since $\catname{FinSet}$ is such a category, species embeddings and species coverings are precisely the epi- and monomorphisms of the functor category $\catname{Spc}$.
  Species isomorphisms are of course the categorical isomorphisms in $\catname{Spc}$.
\end{note}

In the case that $F$ and $G$ are isomorphic species, we will often simply write $F = G$, since they are combinatorially equivalent; some authors instead use $F \simeq G$, reserving the notation of equality for the much stricter case that additionally requires that $F \sbrac{A} = G \sbrac{A}$ as sets for all $A$.
The notions of species embedding and species covering are original to this work.

\begin{example}
  In the motivating examples from above:
  \begin{itemize}
  \item The species $\mathfrak{a}$ of trees \emph{embeds} in the species $\specname{G}$ of graphs by the map which identifies each tree with itself as a graph, since any two distinct trees are distinct as graphs.
  \item The species $\specname{BC}$ of bicolored graphs \emph{covers} the species $\specname{BP}$ of bipartite graphs by the map which sends each bicolored graph to its underlying bipartite graph, since every bipartite graph has at least one bicoloring.
  \item The species $\specname{G}$ of graphs is \emph{isomorphic} with the species $\specname{E} \pbrac{\specname{G}^{\specname{C}}}$ of sets of connected graphs by the map which identifies each graph with its set of connected components, since this decomposition exists uniquely.
  \end{itemize}
\end{example}

\section{Cycle indices and species enumeration}\label{s:cycind}
In classical enumerative combinatorics, formal power series known as `generating functions' are used extensively for keeping track of enumerative data.
In this spirit, we now associate to each species a formal power series which counts structures with respect to their automorphisms, which will prove to be significantly more powerful:
\begin{definition}
  \label{def:cycind}
  For a species $F$, define its \emph{cycle index series} to be the power series
  \begin{equation}
    \label{eq:cycinddef}
    \civars{F}{p_{1}, p_{2}, \dots} := \sum_{n \geq 0} \frac{1}{n!} \big( \sum_{\sigma \in \symgp{n}} \fix \pbrac{F \sbrac{\sigma}} p_{1}^{\sigma_{1}} p_{2}^{\sigma_{2}} \dots \big) = \sum_{n \geq 0} \frac{1}{n!} \big( \sum_{\sigma \in \symgp{n}} \fix \pbrac{F \sbrac{\sigma}} p_{\sigma} \big)
  \end{equation}
  where $\fix \pbrac{F \sbrac{\sigma}} := \abs{\cbrac{s \in F \sbrac{A} : F \sbrac{\sigma} \pbrac{s} = s}}$, where $\sigma_{i}$ is the number of $i$-cycles of $\sigma$, and where $p_{i}$ are indeterminates.
  (That is, $\fix \pbrac{F \sbrac{\sigma}}$ is the \emph{number} of $F$-structures fixed under the action of the transport of $\sigma$.)
  We will make extensive use of the compressed notation $p_{\sigma} = p_{1}^{\sigma_{1}} p_{2}^{\sigma_{2}} \dots$ hereafter.
\end{definition}

In fact, by functorality, $\fix \pbrac{F \sbrac{\sigma}}$ is a class function\footnote{That is, the value of $\fix \pbrac{F \sbrac{\sigma}}$ will be constant on conjugacy classes of permutations, which we note are exactly the sets of permutations of fixed cycle type.} on permutations $\sigma \in \symgp{n}$.
Accordingly, we can instead consider all permutations of a given cycle type at once.
It is a classical theorem that conjugacy classes of permutations in $\symgp{n}$ are indexed by partitions $\lambda \vdash n$, which are defined as multisets of natural numbers whose sum is $n$.
In particular, conjugacy classes are determined by their cycle type, the multiset of the lengths of the cycles, which may clearly be identified bijectively with partitions of $n$.
For a given partition $\lambda \vdash n$, there are $n! / z_{\lambda}$ permutations of cycle type $\lambda$, where $z_{\lambda} := \prod_{i} i^{\lambda_{i}} \lambda_{i}!$ where $\lambda_{i}$ denotes the multiplicity of $i$ in $\lambda$..
Thus, we can instead write
\begin{equation}
  \label{eq:cycinddefpart}
  \civars{F}{p_{1}, p_{2}, \dots} := \sum_{n \geq 0} \sum_{\lambda \vdash n} \fix \pbrac{F \sbrac{\lambda}} \frac{p_{1}^{\lambda_{1}} p_{2}^{\lambda_{2}} \dots}{z_{\lambda}} = \sum_{n \geq 0} \sum_{\lambda \vdash n} \fix \pbrac{F \sbrac{\lambda}} \frac{p_{\lambda}}{z_{\lambda}}
\end{equation}
for $\fix F \sbrac{\lambda} := \fix F \sbrac{\sigma}$ for some choice of a permutation $\sigma$ of cycle type $\lambda$.
Again, we will make extensive use of the notation $p_{\lambda} = p_{\sigma}$ hereafter.

That the cycle index $\ci{F}$ usefully characterizes the enumerative structure of the species $F$ may not be clear.
However, as the following theorems show, both labeled and unlabeled enumeration are immediately possible once the cycle index is in hand.
Recall that, for a given sequence $a = \pbrac{a_{0}, a_{1}, a_{2}, \dots}$, the \emph{ordinary generating function}\footnote{Although these are called `functions' for historical reasons, convergence of these formal power series is often not of immediate interest.} of $a$ is the formal power series $\tilde{A} \pbrac{x} = \sum_{i = 0}^{\infty} a_{i} x^{i}$ and the \emph{exponential generating function}  is the formal power series $A \pbrac{x} = \sum_{i = 1}^{\infty} \frac{1}{i!} a_{i} x^{i}$.
The scaling factor of $\frac{1}{n!}$ in the exponential generating function is convenient in many contexts; for example, it makes differentiation of the generating function a combinatorially-significant operation.
The cycle index of a species is then directly related to two important generating functions:
\begin{theorem}\label{thm:ciegf}
  The exponential generating function $F \pbrac{x}$ of labeled $F$-structures is given by
  \begin{equation}\label{eq:ciegf}
    F \pbrac{x} = \civars{f}{x, 0, 0, \dots}.
  \end{equation}
\end{theorem}
\begin{theorem}\label{thm:ciogf}
  The ordinary generating function $\tilde{F} \pbrac{x}$ of unlabeled $F$-structures is given by
  \begin{equation}\label{eq:ciogf}
    \tilde{F} \pbrac{x} = \ci{F} \pbracs[big]{x, x^{2}, x^{3}, \dots}.
  \end{equation}
\end{theorem}
Proofs of both theorems are found in \cite[\S 1.2]{bll:species}.
In essence, \cref{eq:ciegf} counts each labeled structure exactly once (as a fixed point of the trivial automorphism on $\sbrac{n}$) with a factor of $1/n!$, while \cref{eq:ciogf} simply counts orbits \foreign{\`{a} la} Burnside's Lemma.
In cases where the unlabeled enumeration problem is interesting, it is generally more challenging than the labeled enumeration of the same structures, since the characterization of isomorphism in a species may be nontrivial to capture in a generating function.
If, however, we can calculate the complete cycle index for a species, both labeled and unlabeled enumerations immediately follow.

The use of $p_{i}$ for the variables instead of the more conventional $x_{i}$ alludes to the theory of symmetric functions, in which $p_{i}$ denotes the power-sum functions $p_{i} = \sum_{j} x_{j}^{i}$, which form an important basis for the ring $\Lambda$ of symmetric functions.
When the $p_{i}$ are understood as symmetric functions rather than simply indeterminates, additional P\'{o}lya-theoretic enumerative information is exposed.
In particular, the symmetric function in $x$-variables underlying a cycle index in $p$-variables may be said to count \emph{partially}-labeled structures of a given species, where the coefficient on a monomial $\prod x_{i}^{\alpha_{i}}$ counts structures with $\alpha_{i}$ labels of each sort $i$.
This serves to explain why the coefficients of powers of $p_{1} = \sum_{i} x_{i}$ counts labeled structures (where the labels must all be distinct) and why the automorphism types of structures are enumerated by $\civars{f}{x, x^{2}, x^{3}, \cdots}$, which allows clusters of labels to be the same.
Another application of the theory of symmetric functions to the cycle indices of species may be found in \cite{gessel:laginvspec}.

A more detailed exploration of the history of cycle index polynomials and their relationship to classical P\'{o}lya theory may be found in \cite{jili:pointdet}.

Of course, it is not always obvious how to calculate the cycle index of a species directly.
However, in cases where we can decompose a species as some combination of simpler ones, we can exploit these relationships algebraically to study the cycle indices, as we will see in the next section.

\section{Algebra of species}\label{s:specalg}
It is often natural to describe a species in terms of combinations of other, simpler species---for example, `a permutation is a set of cycles' or `a rooted tree is a single vertex together with a (possibly empty) set of rooted trees'.
Several combinatorial operations on species of structures are commonly used to represent these kinds of combinations; that they have direct analogues in the algebra of cycle indices is in some sense the conceptual justification of the theory.
In particular, for species $F$ and $G$, we will define species $F + G$, $F \cdot G$, $F \circ G$, $\pointed{F}$, and $F'$, and we will compute their cycle indices in terms of $\ci{F}$ and $\ci{G}$.
In what follows, we will not say explicitly what the effects of a given species operation are on bijections when those effects are obvious (as is usually the case).

\begin{definition}\label{def:specsum}
  For two species $F$ and $G$, define their \emph{sum} to be the species $F + G$ given by $\pbracs[big]{F + G} \sbrac{A} = F \sbrac{A} \sqcup G \sbrac{A}$ (where $\sqcup$ denotes disjoint set union).
\end{definition}
In other words, an $\pbrac{F + G}$-structure is an $F$-structure \emph{or} a $G$-structure.
We use the disjoint union to avoid the complexities of dealing with cases where $F \sbrac{A}$ and $G \sbrac{A}$ overlap as sets.

\begin{theorem}\label{thm:specsumci}
  For species $F$ and $G$, the cycle index of their sum is
  \begin{equation}
    \label{eq:specsumci}
    \ci{F + G} = \ci{F} + \ci{G}.
  \end{equation}
\end{theorem}

In the case that $F = G_{1} + G_{2}$, we can simply invert the equation and write $F - G_{2} = G_{1}$.
However, we may instead wish to study the species $F - G$ without first writing $F$ as a sum.
In the spirit of the definition of species addition, we wish to define the species subtraction $F - G$ as the species of $F$-structures that `are not' $G$-structures.
For slightly more generality, we may apply the notions of \cref{def:specmaptypes}:
\begin{definition}
  \label{def:specdif}
  For two species $F$ and $G$ with a species embedding $\phi: G \to F$, define their \emph{difference with respect to $\phi$} to be the species $F \specsub{\phi} G$ given by $\pbracs[big]{F \specsub{\phi} G} \sbrac{A} \defeq F \sbrac{A} - \phi \pbrac{G \sbrac{A}}$.
  When there is no ambiguity about the choice of embedding $\phi$, especially in the case that $G$ has a combinatorially natural embedding in $F$, we may instead simply write $F - G$ and call this species their \emph{difference}.
\end{definition}

For example, for $\specname{G}$ the species of graphs and $\mathfrak{a}$ the species of trees with the natural embedding, we have that $\specname{G} - \mathfrak{a}$ is the species of graphs with cycles.

We note also that species addition is associative and commutative (up to species isomorphism), and furthermore the empty species $\numspecname{0}: A \mapsto \varnothing$ is an additive identity, so species with addition form an abelian monoid.
This can be completed to create the abelian group of \emph{virtual species}, in which the subtraction $F - G$ of arbitrary species is defined; the two definitions in fact agree where our definition applies.
We will not delve into the details of virtual species theory here, directing the reader instead to \cite[\S 2.5]{bll:species}.

\begin{definition}\label{thm:specprod}
  For two species $F$ and $G$, define their \emph{product} to be the species $F \cdot G$ given by $\pbrac{F \cdot G} \sbrac{A} = \sum_{A = B \sqcup C} F \sbrac{B} \times G \sbrac{C}$.
\end{definition}
In other words, an $\pbrac{F \cdot G}$-structure is a partition of $A$ into two sets $B$ and $C$, an $F$-structure on $B$, and a $G$-structure on $C$.
This definition is partially motivated by the following result on cycle indices:
\begin{theorem}\label{thm:specprodci}
  For species $F$ and $G$, the cycle index of their product is
  \begin{equation}
    \label{eq:specprodci}
    \ci{F \cdot G} = \ci{F} \cdot \ci{G}.
  \end{equation}
\end{theorem}

Conceptually, the species product can be used to describe species that decompose uniquely into substructures of two specified species.
For example, a permutation on a set $A$ decomposes uniquely into a (possibly empty) set of fixed points and a derangement of their complement in $A$.
Thus, $\specname{S} = \specname{E} \cdot \operatorname{Der}$ for $\specname{S}$ the species of permutations, $\specname{E}$ the species of sets, and $\operatorname{Der}$ the species of derangements.

We note also that species multiplication is commutative (up to species isomorphism) and distributes over addition, so the class of species with addition and multiplication forms a commutative semiring, with the species $\numspecname{1}: \begin{cases} \varnothing \mapsto \cbrac{\varnothing} \\ A \neq \varnothing \mapsto \varnothing \end{cases}$ as a multiplicative identity; if addition is completed as previously described, the class of virtual species with addition and multiplication forms a true commutative ring.

In addition, the question of which species can be decomposed as sums and products without resorting to virtual species is one of great interest; the notions of \emph{molecular} and \emph{atomic} species are directly derived from such decompositions, and represent the beginnings of the systematic study of the structure of the class of species as a whole.
Further details on this topic are presented in \cite[\S 2.6]{bll:species}.

\begin{definition}
  \label{def:speccomp}
  For two species $F$ and $G$ with $G \sbrac{\varnothing} = \varnothing$, define their \emph{composition} to be the species $F \circ G$ given by $\pbrac{F \circ G} \sbrac{A} = \prod_{\pi \in P \pbrac{A}} \pbrac{F \sbrac{\pi} \times \prod_{B \in \pi} G \sbrac{B}}$ where $P \pbrac{A}$ is the set of partitions of $A$.
\end{definition}
In other words, the composition $F \circ G$ produces the species of $F$-structures of collections of $G$-structures.
The definition is, again, motivated by a correspondence with a certain operation on cycle indices:
\begin{definition}
  \label{def:cipleth}
  Let $f$ and $g$ be cycle indices. Then the \emph{plethysm} $f \circ g$ is the cycle index
  \begin{equation}
    \label{eq:cipleth}
    f \circ g = f \pbrac{g \pbrac{p_{1}, p_{2}, p_{3}, \dots}, g \pbrac{p_{2}, p_{4}, p_{6}, \dots}, \dots},
  \end{equation}
  where $f \pbrac{a, b, \dots}$ denotes the cycle index $f$ with $a$ substituted for $p_{1}$, $b$ substituted for $p_{2}$, and so on.
\end{definition}
This definition is inherited directly from the theory of symmetric functions in infinitely many variables, where our $p_{i}$ are basis elements, as previously discussed. This operation on cycle indices then corresponds exactly to species composition:
\begin{theorem}
  \label{thm:speccompci}
  For species $F$ and $G$ with $G \sbrac{\varnothing} = \varnothing$, the cycle index of their plethysm is
  \begin{equation}
    \label{eq:speccompci}
    \ci{F \circ G} = \ci{F} \circ \ci{G}
  \end{equation}
  where $\circ$ in the right-hand side is as in \cref{eq:cipleth}.
\end{theorem}
Many combinatorial structures admit natural descriptions as compositions of species.
For example, every graph admits a unique decomposition as a (possibly empty) set of (nonempty) connected graphs, so we have the species identity $\specname{G} = \specname{E} \circ \specname{G}^{C}$ for $\specname{G}$ the species of graphs and $\specname{G}^{C}$ the species of nonempty connected graphs.

Diligent readers may observe that the requirement that $G \sbrac{\varnothing} = \varnothing$ in \cref{def:speccomp} is in fact logically vacuous, since the given construction would simply ignore the $\varnothing$-structures.
However, the formula in \cref{thm:speccompci} fails to be well-defined for any $\ci{G}$ with non-zero constant term (corresponding to species $G$ with nonempty $G \sbrac{\varnothing}$) unless $\ci{F}$ has finite degree (corresponding to species $F$ with support only in finitely many degrees).
Consider the following example:
\begin{example}
  Let $\specname{E}$ denote the species of sets, $\specname{E}_{3}$ its restriction to sets with three elements, $\numspecname{1}$ the species described above (which has one empty structure), and $X$ the species of singletons (which has one order-$1$ structure).
  If $\specname{E} \pbrac{\numspecname{1} + X}$ were well-defined, it would denote the species of `partially-labeled sets'.
  However, for fixed cardinality $n$, there is an $\specname{E} \pbrac{\numspecname{1} + X}$-structure on $n$ labels \emph{for each nonnegative $k$}---specifically, the set $\sbrac{n}$ together with $k$ unlabeled elements.
  Thus, there would be infinitely many structures of each cardinality for this `species', so it is not in fact a species at all.

  However, the situation for $\specname{E}_{3} \pbrac{\numspecname{1} + X}$ is entirely different.
  A structure in this species is a $3$-set, some of whose elements are labeled.
  There are only four possible such structures: $\cbrac{*, *, *}$, $\cbrac{*, *, 1}$, $\cbrac{*, 1, 2}$, and $\cbrac{1, 2, 3}$, where $*$ denotes an unlabeled element and integers denote labeled elements.
  Moreover, by discarding the unlabeled elements, we can clearly see that $\specname{E}_{3} \pbrac{\numspecname{1} + X} = \sum_{i = 0}^{3} \specname{E}_{i}$.
\end{example}
In our setting, we will not use this alternative notion of composition, so we will not develop it formally here.

Several other binary operations on species are defined in the literature, including the Cartesian product $F \times G$, the functorial composition $F \square G$, and the inner plethysm $F \boxtimes G$ of \cite{travis:inpleth}.
We will not use these here.
However, we do introduce two unary operations: $\pointed{F}$ and $F'$.

\begin{definition}
  \label{def:specderiv}
  For a species $F$, define its \emph{species derivative} to be the species $\deriv{F}$ given by $\deriv{F} \sbrac{A} = F \sbrac{A \cup \cbrac{*}}$ for $*$ an element chosen not in $A$ (say, the set $A$ itself).
\end{definition}
It is important to note that the label $*$ of an $\deriv{F}$-structure is \emph{distinguished} from the other labels; the automorphisms of the species $\deriv{F}$ cannot interchange $*$ with another label.
Thus, species differentiation is appropriate for cases where we want to remove one `position' in a structure.
For example, for $\specname{L}$ the species of linear orders and $\specname{C}$ the species of cyclic orders, we have $\specname{L} = \deriv{\specname{C}}$; a cyclic order on the set $A \cup \cbrac{*}$ is naturally associated with the linear order on the set $A$ produced by removing $*$.
Terming this operation `differentiation' is justified by its effect on cycle indices:
\begin{theorem}
  \label{thm:specderivci}
  For a species $F$, the cycle index of its derivative is given by
  \begin{equation}
    \label{eq:specderivci}
    \civars{\deriv{F}}{p_{1}, p_{2}, \dots} = \frac{\partial}{\partial p_{1}} \civars{F}{p_{1}, p_{2}, \dots}.
  \end{equation}
\end{theorem}
We note that we cannot in general recover $\ci{F}$ from $\ci{\deriv{F}}$, since there may be terms in $\ci{F}$ which have no $p_{1}$-component (corresponding to $F$-structures which have no automorphisms with fixed points).

Finally, we introduce a variant of the species derivative which allows us to \emph{label} the distinguished element $*$:
\begin{definition}
  \label{def:specpoint}
  For a species $F$, define its \emph{pointed species} to be the species $\pointed{F}$ given by $\pointed{F} \sbrac{A} = F \sbrac{A} \times A$ (that is, pairs of the form $\pbrac{f, a}$ where $f$ is an $F$-structure on $A$ and $a \in A$) with transport $\pointed{F} \sbrac{\sigma} \pbrac{f, a} = \pbrac{F \sbrac{\sigma} \pbrac{f}, \sigma \pbrac{a}}$.
  We can also write $\pointed{F} \sbrac{A} = X \cdot \deriv{F}$ for $X$ the species of singletons.
\end{definition}
In other words, an $\pointed{F} \sbrac{A}$-structure is an $F \sbrac{A}$-structure with a distinguished element taken from the set $A$ (as opposed to $\deriv{F}$, where the distinguished element is new).
Thus, species pointing is appropriate for cases such as those of rooted trees: for $\mathfrak{a}$ the species of trees and $\specname{A}$ the species of rooted trees, we have $\specname{A} = \pointed{\mathfrak{a}}$.
\Cref{eq:specderivci} leads directly to the following:
\begin{theorem}
  \label{thm:specpointci}
  For a species $F$, the cycle index of its corresponding pointed species is given by
  \begin{equation}
    \label{eq:specpointci}
    \ci{\pointed{F}} = \ci{X} \cdot \ci{\deriv{F}}.
  \end{equation}
\end{theorem}
Note that, again, we cannot in general recover $\ci{F}$ from $\ci{\pointed{F}}$, for the same reasons as in the case of $\ci{\deriv{F}}$.

\section{Multisort species}\label{s:mult}
A species $F$ as defined in \cref{def:species} is a functor $F: \catname{FinBij} \to \catname{FinSet}$; an $F$-structure in $F \sbrac{A}$ takes its labels from the set $A$.
The tool-set so produced is adequate to describe many classes of combinatorial structures.
However, there is one particular structure type which it cannot effectively capture: the notion of distinct \emph{sorts} of elements within a structure.
Perhaps the most natural example of this is the case of $k$-colored graphs, where every vertex has one of $k$ colors with the requirement that no pair of adjacent vertices shares a color.
Automorphisms of such a graph must preserve the colorings of the vertices, which is not a natural restriction to impose in the calculation of the classical cycle index in \cref{eq:cycinddef}.
We thus incorporate the notion of sorts directly into a new definition:
\begin{definition} 
  \label{def:ksortset}
  For a fixed integer $k \geq 1$, define a \emph{$k$-sort set} to be an ordered $k$-tuple of sets.
  Say that a $k$-sort set is \emph{finite} if each component set is finite; in that case, its \emph{$k$-sort cardinality} is the ordered tuple of its components' set cardinalities.
  Further, define a \emph{$k$-sort function} to be an ordered $k$-tuple of set functions which acts componentwise on $k$-sort sets.
  For two $k$-sort sets $U$ and $V$, a $k$-sort function $\sigma$ is a \emph{$k$-sort bijection} if each component is a set bijection.
  For $k$-sort sets of cardinality $\pbrac{c_{1}, c_{2}, \dots, c_{k}}$, denote by $\symgp{c_{1}, c_{2}, \dots, c_{k}} = \symgp{c_{1}} \times \symgp{c_{2}} \times \dots \times \symgp{c_{k}}$ the \emph{$k$-sort symmetric group}, the elements of which are in natural bijection with $k$-sort bijections from a $k$-sort set to itself.
  Finally, denote by $\catname{FinBij}^{k}$ the category of finite $k$-sort sets with $k$-sort bijections.
\end{definition}

We can then define an extension of species to the context of $k$-sort sets:
\begin{definition}
  \label{def:multisort}
  A \emph{$k$-sort species} $F$ is a functor $F: \catname{FinBij}^{k} \to \catname{FinBij}$ which associates to each $k$-sort set $U$ a set $F \sbrac{U}$ of \emph{$k$-sort $F$-structures} and to each $k$-sort bijection $\sigma: U \to V$ a bijection $F \sbrac{\sigma}: F \sbrac{U} \to F \sbrac{V}$.
\end{definition}
Functorality once again imposes naturality conditions on these associations.

Just as in the theory of ordinary species, to each multisort species is associated a power series, its \emph{cycle index}, which carries essential combinatorial data about the automorphism structure of the species.
To keep track of the multiple sorts of labels, however, we require multiple sets of indeterminates.
Where in ordinary cycle indices we simply used $p_{i}$ for the $i$th indeterminate, we now use $p_{i} \sbrac{j}$ for the $i$th indeterminate of the $j$th sort.
In some contexts with small $k$, we will denote our sorts with letters (saying, for example, that we have `$X$ labels' and `$Y$ labels'), in which case we will write $p_{i} \sbrac{x}$, $p_{i} \sbrac{y}$, and so forth.
In natural analogy to \cref{def:cycind}, the formula for the cycle index of a $k$-sort species $F$ is given by
\begin{multline}
  \label{eq:multcycinddef}
  \civars{F}{p_{1} \sbrac{1}, p_{2} \sbrac{1}, \dots; p_{1} \sbrac{2}, p_{2} \sbrac{2}, \dots; \dots; p_{1} \sbrac{k}, p_{2} \sbrac{k}, \dots} = \\
  \sum_{\substack{n \geq 0 \\ a_{1} + a_{2} + \dots + a_{k} = n}} \frac{1}{a_{1}! a_{2}! \dots a_{k}!} \sum_{\sigma \in \symgp{a_{1}, a_{2}, \dots, a_{k}}} \fix \pbrac{F \sbrac{\sigma}} p^{\sigma_{1}}_{\sbrac{1}} p^{\sigma_{2}}_{\sbrac{2}} \dots p^{\sigma_{k}}_{\sbrac{k}}.
\end{multline}
where by $p^{\sigma_{i}}_{\sbrac{i}}$ we denote the product $\prod_{j} \pbrac{p_{j} \sbrac{i}}^{\pbrac{\sigma_{i}}_{j}}$ where $\pbrac{\sigma_{i}}_{j}$ is the number of $j$-cycles of $\sigma_{i}$.

The operations of addition and multiplication extend to the multisort context naturally.
To make sense of differentiation and pointing, we need only specify a sort from which to draw the element or label which is marked; we then write $\deriv[X]{F}$ and $\pointed[X]{F}$ for the derivative and pointing respectively of $F$ `in the sort $X$', which is to say with its distinguished element drawn from that sort.
When $F$ is a $1$-sort species and $G$ a $k$-sort species, the construction of the $k$-sort species $F \circ G$ is natural; in other settings, we will not define a general notion of composition of multisort species.

\section{$\Gamma$-species and quotient species}\label{s:quot}
It is frequently the case that interesting combinatorial problems admit elegant descriptions in terms of quotients of a class of structures $F$ under the action of a group $\Gamma$.
In some cases, this group action will be \emph{structural} in the sense that it commutes with permutations of labels in the species $F$, or, informally, that it is independent of the choice of labelings on each $F$-structure.
In such a case, we may also say that $\Gamma$ acts on `unlabeled structures' of the class $F$.

\begin{example}
  \label{ex:graphcomp}
  Let $\specname{G}$ denote the species of simple graphs.
  Let the group $\symgp{2}$ act on such graphs by letting the identity act trivially and letting the non-trivial element $\pmt{(12)}$ send each graph to its complement (that is, by replacing each edge with a non-edge and each non-edge with an edge).
  This `complementation action' is structural in the sense described previously.
\end{example}

We note that a group action is structural is exactly the condition that each $\gamma \in \Gamma$ acts by a species isomorphism $\gamma: F \to F$ in the sense of \cref{def:specmaptypes}.

We now incorporate such species-compatible actions into a new definition:
\begin{definition}
  \label{def:gspecies}
  For $\Gamma$ a group, a \emph{$\Gamma$-species} $F$ is a combinatorial species $F$ together with an action of $\Gamma$ on $F$-structures by species isomorphisms.
  Explicitly, for $F$ a $\Gamma$-species, the diagram
  \begin{equation*}
    \begin{tikzpicture}[every node/.style={fill=white}]
      \matrix (m) [matrix of math nodes, row sep=4em, column sep=4em, text height=1.5ex, text depth=0.25ex]
      {
        A & F \sbrac{A} & F \sbrac{A} \\
        B & F \sbrac{B}  & F \sbrac{B} \\
      };
      \path[->,font=\scriptsize]
      (m-1-1) edge node {$F$} (m-1-2)
      edge node {$\sigma$} (m-2-1)
      (m-1-2) edge node {$\gamma_{A}$} (m-1-3)
      edge node {$F \sbrac{\sigma}$} (m-2-2)
      (m-1-3) edge node {$F \sbrac{\sigma}$} (m-2-3)
      (m-2-1) edge node {$F$} (m-2-2)
      (m-2-2) edge node {$\gamma_{B}$} (m-2-3);
    \end{tikzpicture}
  \end{equation*}
  commutes for every $\gamma \in \Gamma$ and every set bijection $\sigma: A \to B$.
  (Note that commutativity of the left square is required for $F$ to be a species at all.)
\end{definition}

$\specname{G}$ is then a $\symgp{2}$-species with the action described in \cref{ex:graphcomp}.

For such a $\Gamma$-species, of course, it is then meaningful to pass to the quotient under the action by $\Gamma$:
\begin{definition}
  \label{def:qspecies}
  For $F$ a $\Gamma$-species, define $\nicefrac{F}{\Gamma}$, the \emph{quotient species} of $F$ under the action of $\Gamma$, to be the species of $\Gamma$-orbits of $F$-structures.
\end{definition}

\begin{example}
  \label{ex:graphcompquot}
  Consider $\specname{G}$ as a $\symgp{2}$-species in light of the action defined in \cref{ex:graphcomp}.
  The structures of the quotient species $\nicefrac{\specname{G}}{\symgp{2}}$ are then pairs of complementary graphs.
  We may choose to interpret each such pair as representing a $2$-partition of the set of vertex pairs of the complete graph (that is, of edges of the complete graph).
  More natural examples of quotient structures will present themselves in later chapters.
\end{example}

For each label set $A$, let $\quomap{\Gamma} \sbrac{A}: F \sbrac{A} \to \nicefrac{F}{\Gamma} \sbrac{A}$ denote the map sending each $F$-structure over $A$ to its quotient $\nicefrac{F}{\Gamma}$-structure over $A$.
Then $\quomap{\Gamma} \sbrac{A}$ is an injection for each $A$, and the requirement that $\Gamma$ acts by natural transformations implies that the induced functor map $\quomap{\Gamma}: F \to \nicefrac{F}{\Gamma}$ is a natural transformation.
Thus, the passage from $F$ to $\nicefrac{F}{\Gamma}$ is a species cover in the sense of \cref{def:specmaptypes}.

A brief exposition of the notion of quotient species may be found in \cite[\S 3.6]{bll:species}, and a more thorough exposition (in French) in \cite{bous:species}.
Our motivation, of course, is that combinatorial structures of a given class are often `naturally' identified with orbits of structures of another, larger class under the action of some group.
Our goal will be to compute the cycle index of the species $\nicefrac{F}{\Gamma}$ in terms of that of $F$ and information about the $\Gamma$-action, so that enumerative data about the quotient species can be extracted.

As an intermediate step to the computation of the cycle index associated to this quotient species, we associate a cycle index to a $\Gamma$-species $F$ that keeps track of the needed data about the $\Gamma$-action.
\begin{definition}
  \label{def:gcycind}
  For a $\Gamma$-species $F$, define the $\Gamma$-cycle index $\gci{\Gamma}{F}$ as in \cite{hend:specfield}: for each $\gamma \in \Gamma$, let
  \begin{equation}
    \gcivars{\Gamma}{F}{\gamma} = \sum_{n \geq 0} \frac{1}{n!} \sum_{\sigma \in \symgp{n}} \fix \pbrac{\gamma \cdot F \sbrac{\sigma}} p_{\sigma} \label{eq:gcycinddef}
  \end{equation}
  with $p_{\sigma}$ as in \cref{eq:cycinddef}.
\end{definition}

We will call such an object (formally a map from $\Gamma$ to the ring $\ringname{Q} \sbrac{\sbrac{p_{1}, p_{2}, \dots}}$ of symmetric functions with rational coefficients in the $p$-basis) a \emph{$\Gamma$-cycle index} even when it is not explicitly the $\Gamma$-cycle index of a $\Gamma$-species, and we will sometimes call $\gcielt{\Gamma}{F}{\gamma}$ the ``$\gamma$ term of $\gci{\Gamma}{F}$''.
So the coefficients in the power series count the fixed points of the \emph{combined} action of a permutation and the group element $\gamma$.
Note that, in particular, the classical (`ordinary') cycle index may be recovered as $\ci{F} = \gcielt{\Gamma}{F}{e}$ for any $\Gamma$-species $F$.

The algebraic relationships between ordinary species and their cycle indices generally extend without modification to the $\Gamma$-species context, as long as appropriate allowances are made.
The actions on cycle indices of $\Gamma$-species addition and multiplication are exactly as in the ordinary species case considered componentwise:
\begin{definition}
  \label{def:gspecsumprod}
  For two $\Gamma$-species $F$ and $G$, the $\Gamma$-cycle index of their sum $F + G$ is given by
  \begin{equation}
    \label{eq:gspecsum}
    \gcielt{\Gamma}{F + G}{\gamma} = \gcielt{\Gamma}{F}{\gamma} + \gcielt{\Gamma}{G}{\gamma}
  \end{equation}
  and the $\Gamma$-cycle index of their product $F \cdot G$ is given by
  \begin{equation}
    \label{eq:gspecprod}
    \gcielt{\Gamma}{F \cdot G}{\gamma} = \gcielt{\Gamma}{F}{\gamma} \cdot \gcielt{\Gamma}{G}{\gamma}
  \end{equation}
\end{definition}
The action of composition, which in ordinary species corresponds to plethysm of cycle indices, can also be extended:
\begin{definition}
  \label{def:gspeccomp}
  For two $\Gamma$-species $F$ and $G$, define their \emph{composition} to be the $\Gamma$-species $F \circ G$ with structures given by $\pbrac{F \circ G} \sbrac{A} = \prod_{\pi \in P \pbrac{A}} \pbrac{F \sbrac{\pi} \times \prod_{B \in \pi} G \sbrac{B}}$ where $P \pbrac{A}$ is the set of partitions of $A$ and where $\gamma \in \Gamma$ acts on a $\pbrac{F \circ G}$-structure by acting on the $F$-structure and the $G$-structures independently.
\end{definition}
The requirement in \cref{def:gspecies} that the action of $\Gamma$ commutes with transport implies that this is well-defined.
Informally, for $\Gamma$-species $F$ and $G$, we have defined the composition $F \circ G$ to be the $\Gamma$-species of $F$-structures of $G$-structures, where $\gamma \in \Gamma$ acts on an $\pbrac{F \circ G}$-structure by acting independently on the $F$-structure and each of its associated $G$-structures.
A formula similar to that \cref{thm:speccompci} requires a definition of the plethysm of $\Gamma$-symmetric functions, here taken from \cite[\S 3]{hend:specfield}:
\begin{definition}
  \label{def:gcipleth}
  For two $\Gamma$-cycle indices $f$ and $g$, their \emph{plethysm} $f \circ g$ is a $\Gamma$-cycle index defined by
  \begin{equation}
    \pbrac{f \circ g} \pbrac{\gamma} = f \pbrac{\gamma} \pbrac{g \pbrac{\gamma} \pbrac{p_{1}, p_{2}, p_{3}, \dots}, g \pbracs[big]{\gamma^{2}} \pbrac{p_{2}, p_{4}, p_{6}, \dots}, \dots}.
    \label{eq:gcipleth}
  \end{equation}
\end{definition}
This definition of $\Gamma$-cycle index plethysm is then indeed the correct operation to pair with the composition of $\Gamma$-species:
\begin{theorem}[Theorem 3.1, \cite{hend:specfield}]
  \label{thm:gspeccompci}
  If $A$ and $B$ are $\Gamma$-species and $B \pbrac{\varnothing} = \varnothing$, then
  \begin{equation}
    \label{eq:gspeccompci}
    \gci{\Gamma}{A \circ B} = \gci{\Gamma}{A} \circ \gci{\Gamma}{B}.
  \end{equation}
\end{theorem}
Thus, $\Gamma$-species admit the same sorts of `nice' correspondences between structural descriptions (in terms of functorial algebra) and enumerative characterizations (in terms of cycle indices) that ordinary species do.

However, to make use of this theory for enumerative purposes, we also need to be able to pass from the $\Gamma$-cycle index of a $\Gamma$-species to the ordinary cycle index of its associated quotient species under the action of $\Gamma$.
This will allow us to adopt a useful strategy: if we can characterize some difficult-to-enumerate combinatorial structure as quotients of more accessible structures, we will be able to apply the full force of species theory to the enumeration of the prequotient structures, \emph{then} pass to the quotient when it is convenient.
Exactly this approach will serve as the core of both of the following chapters.

Since we intend to enumerate orbits under a group action, we apply a generalization of Burnside's Lemma found in \cite[Lemma 5]{gessel:laginvspec}:
\begin{lemma}
  \label{lem:grouporbits}
  If $\Gamma$ and $\Delta$ are finite groups and $S$ a set with a $\pbrac{\Gamma \times \Delta}$-action, for any $\delta \in \Delta$ the number of $\Gamma$-orbits fixed by $\delta$ is $\frac{1}{\abs{\Gamma}} \sum_{\gamma \in \Gamma} \fix \pbrac{\gamma, \delta}$.
\end{lemma}

Recall from \cref{eq:cycinddef} that, to compute the cycle index of a species, we need to enumerate the fixed points of each $\sigma \in \symgp{n}$.
However, to do this in the quotient species $\nicefrac{F}{\Gamma}$ is by definition to count the fixed $\Gamma$-orbits of $\sigma$ in $F$ under commuting actions of $\symgp{n}$ and $\Gamma$ (that is, under an $\pbrac{\symgp{n} \times \Gamma}$-action).
Thus, \cref{lem:grouporbits} implies the following:
\begin{theorem}\label{thm:qsci}
  For a $\Gamma$-species $F$, the ordinary cycle index of the quotient species $\nicefrac{F}{\Gamma}$ is given by 
  \begin{equation}
    \label{eq:quotcycind}
    \ci{F / \Gamma} = \qgci{\Gamma}{F} \defeq \frac{1}{\abs{\Gamma}} \sum_{\gamma \in \Gamma} \gcielt{\Gamma}{F}{\gamma} = \frac{1}{\abs{\Gamma}} \sum_{\substack{n \geq 0 \\ \sigma \in \symgp{n} \\ \gamma \in \Gamma}} \frac{1}{n!} \pbrac{\gamma \cdot F \sbrac{\sigma}} p_{\sigma}.
  \end{equation}
  where we define $\qgci{\Gamma}{F} = \frac{1}{\abs{\Gamma}} \sum_{\gamma \in \Gamma} \gcielt{\Gamma}{F}{\gamma}$ for future convenience.
\end{theorem}
Note that this same result on cycle indices is implicit in \cite[\S 2.2.3]{bous:species}.
With it, we can compute explicit enumerative data for a quotient species using cycle-index information of the original $\Gamma$-species with respect to the group action, as desired.

Recall from \cref{thm:ciegf,thm:ciogf} that the exponential generating function $F \pbrac{x}$ of labeled $F$-structures and the ordinary generating function $\tilde{F} \pbrac{x}$ of unlabeled $F$-structures may both be computed from the cycle index $\ci{F}$ of an ordinary species $F$ by simple substitutions.
In the $\Gamma$-species context, we may perform similar substitutions to derive analogous generating functions.

\begin{theorem}
  \label{thm:gciegf}
  The exponential generating function $F_{\gamma} \pbrac{x}$ of labeled $\gamma$-invariant $F$-structures is
  \begin{equation}
    \label{eq:gciegf}
    F_{\gamma} \pbrac{x} = \gcieltvars{\Gamma}{F}{\gamma}{x, 0, 0, \dots}.
  \end{equation}
\end{theorem}

\begin{theorem}
  \label{thm:gciogf}
  The ordinary generating function $\tilde{F}_{\gamma} \pbrac{x}$ of unlabeled $\gamma$-invariant $F$-structures is
  \begin{equation}
    \label{eq:gciogf}
    \tilde{F}_{\gamma} \pbrac{x} = \gcieltvars{\Gamma}{F}{\gamma}{x, x^{2}, x^{3}, \dots}.
  \end{equation}
\end{theorem}

These theorems follow directly from \cref{eq:ciegf,eq:ciogf}, thinking of $F_{\gamma} \pbrac{x}$ and $\widetilde{F_{\gamma} \pbrac{x}}$ as enumerating the combinatorial class of $F$-structures which are invariant under $\gamma$.

Note that the notion of `unlabeled $\gamma$-invariant $F$-structures' is always well-defined precisely because \cref{def:gspecies} requires that the action of $\Gamma$ commutes with transport of structures.

From these results and \cref{thm:qsci}, we can the conclude:
\begin{theorem}
  \label{qgciegf}
  The exponential generating function $F \pbrac{x}$ of labeled $\nicefrac{F}{\Gamma}$-structures is
  \begin{equation}
    \label{eq:qgciegf}
    F \pbrac{x} = \frac{1}{\abs{\Gamma}} \sum_{\gamma \in \Gamma} F_{\gamma} \pbrac{x}.
  \end{equation}
\end{theorem}

Similarly,
\begin{theorem}
  \label{qgciogf}
  The ordinary generating function $\tilde{F} \pbrac{x}$ of unlabeled $\nicefrac{F}{\Gamma}$-structures is
  \begin{equation}
    \label{eq:qgcogf}
    \tilde{F} \pbrac{x} = \frac{1}{\abs{\Gamma}} \sum_{\gamma \in \Gamma} \tilde{F}_{\gamma} \pbrac{x}.
  \end{equation}
\end{theorem}

Note also that all of the above extends naturally into the multisort species context.
We will use this extensively in \cref{c:ktrees}.
It also extends naturally to weighted contexts, but we will not apply this extension here.

\chapter{The species of bipartite blocks}\label{c:bpblocks}
\section{Introduction}\label{s:bpintro}
We first apply the theory of quotient species to the enumeration of bipartite blocks.

\begin{definition}
  \label{def:bcgraph}
  A \emph{bicolored graph} is a graph $\Gamma$ each vertex of which has been assigned one of two colors (here, black and white) such that each edge connects vertices of different colors.
  A \emph{bipartite graph} (sometimes called \emph{bicolorable}) is a graph $\Gamma$ which admits such a coloring.  
\end{definition}

There is an extensive literature about bicolored and bipartite graphs, including enumerative results for bicolored graphs \cite{har:bicolored}, bipartite graphs both allowing \cite{han:bipartite} and prohibiting \cite{harprins:bipartite} isolated points, and bipartite blocks \cite{harrob:bipblocks}.
However, this final enumeration was previously completed only in the labeled case.
By considering the problem in light of the theory of $\Gamma$-species, we develop a more systematic understanding of the structural relationships between these various classes of graphs, which allows us to enumerate all of them in both labeled and unlabeled settings.

Throughout this chapter, we denote by $\specname{BC}$ the species of bicolored graphs and by $\specname{BP}$ the species of bipartite graphs.
The prefix $\specname{C}$ will indicate the connected analogue of such a species.

We are motivated by the graph-theoretic fact that each \emph{connected} bipartite graph may be identified with exactly two bicolored graphs which are color-dual.
In other words, a connected bipartite graph is (by definition or by easy exercise, depending on your approach) an orbit of connected bicolored graphs under the action of $\symgp{2}$ where the nontrivial element $\tau$ reverses all vertex colors.
We will hereafter treat all the various species of bicolored graphs as $\symgp{2}$-species with respect to this action and use the theory developed in \cref{s:quot} to pass to bipartite graphs.

Although the theory of multisort species presented in \cref{s:mult} is in general well-suited to the study of colored graphs, we will not need it here.
The restrictions that vertex colorings place on automorphisms of bicolored graphs are simple enough that we can deal with them directly.

\section{Bicolored graphs}\label{s:bcgraph}
We begin our investigation by directly computing the $\symgp{2}$-cycle index for the species $\specname{BC}$ of bicolored graphs with the color-reversing $\symgp{2}$-action described previously.
We will then use various methods from the species algebra of \cref{c:species} to pass to various other species.

\subsection{Computing $\gcielt{\symgp{2}}{\specname{BC}}{e}$}\label{ss:ecibc}
We construct the cycle index for the species $\specname{BC}$ of bicolored graphs in the classical way, which in light of our $\symgp{2}$-action will give $\gcielt{\symgp{2}}{\specname{BC}}{e}$. 

Recall the formula for the cycle index of a $\Gamma$-species in \cref{eq:gcycinddef}:
\begin{equation*}
  \gcielt{\Gamma}{F}{\gamma} = \sum_{n \geq 0} \frac{1}{n!} \sum_{\sigma \in \symgp{n}} \fix \pbrac{\gamma \cdot F \sbrac{\sigma}} p_{\sigma}.
\end{equation*}
Thus, for each $n > 0$ and each permutation $\pi \in \symgp{n}$, we must count bicolored graphs on $\sbrac{n}$ for which $\pi$ is a color-preserving automorphism.
To simplify some future calculations, we omit empty graphs and define $\specname{BC} \sbrac{\varnothing} = \varnothing$.
We note that the \emph{number} of such graphs in fact depends only on the cycle type $\lambda \vdash n$ of the permutation $\pi$, so we can use the cycle index formula in \cref{eq:cycinddefpart} interpreted as a $\Gamma$-cycle index identity.

Fix some $n \geq 0$ and let $\lambda \vdash n$.
We wish to count bicolored graphs for which a chosen permutation $\pi$ of cycle type $\lambda$ is a color-preserving automorphism.
Each cycle of the permutation must correspond to a monochromatic subset of the vertices, so we may construct graphs by drawing bicolored edges into a given colored vertex set.
If we draw some particular bicolored edge, we must also draw every other edge in its orbit under $\pi$ if $\pi$ is to be an automorphism of the graph.
Moreover, every bicolored graph for which $\pi$ is an automorphism may be constructed in this way
Therefore, we direct our attention first to counting these edge orbits for a fixed coloring; we will then count colorings with respect to these results to get our total cycle index.

Consider an edge connecting two cycles of lengths $m$ and $n$; the length of its orbit under the permutation is $\lcm \pbrac{m, n}$, so the number of such orbits of edges between these two cycles is $mn / \lcm \pbrac{m, n} = \gcd \pbrac{m, n}$.
For an example in the case $m = 4, n = 2$, see \cref{fig:exbcecycle}.
The number of orbits for a fixed coloring is then $\sum \gcd \pbrac{m, n}$ where the sum is over the multiset of all cycle lengths $m$ of white cycles and $n$ of black cycles in the permutation $\pi$.
We may then construct any possible graph fixed by our permutation by making a choice of a subset of these cycles to fill with edges, so the total number of such graphs is $\prod 2^{\gcd \pbrac{m, n}}$ for a fixed coloring.

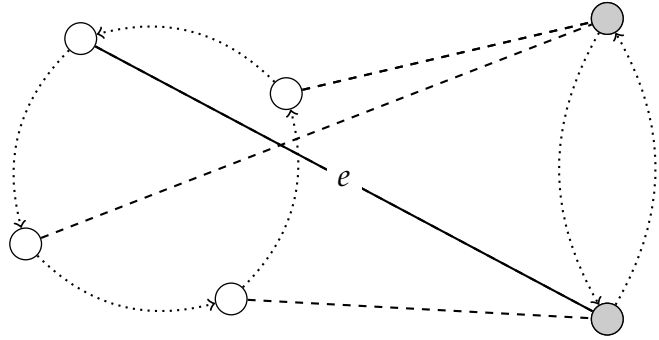
\begin{figure}[htb]
  \centering

  \begin{tikzpicture}
    % \SetGraphUnit{2}
    \GraphInit[vstyle=Hasse]
    
    \begin{scope}[xshift=-3cm,rotate=30]
      \SetUpEdge[style=cycedge]
      \grCycle[RA=2,prefix=a]{4}
    \end{scope}

    \begin{scope}[xshift=+3cm,rotate=90]
      \SetUpEdge[style=cycedge]
      \grCycle[RA=2,prefix=b]{2}
      \AddVertexColor{black!20}{b0,b1}
    \end{scope}

    \SetUpEdge[style=dashed]
    \EdgeDoubleMod{a}{4}{0}{1}{b}{2}{0}{1}{4}
    
    \SetUpEdge[style=solid]
    \Edge[label={$e$}](a1)(b1)
  \end{tikzpicture}
  \caption[Example edge-orbit of a color-preserving automorphism]{An edge $e$ (solid) between two cycles of lengths $4$ and $2$ in a permutation and that edge's orbit (dashed)}
  \label{fig:exbcecycle}
\end{figure}

We now turn our attention to the possible colorings of the graph which are compatible with a permutation of specified cycle type $\lambda$.
We split our partition into two subpartitions, writing $\lambda = \mu \cup \nu$, where partitions are treated as multisets and $\cup$ is the multiset union, and designate $\mu$ to represent the white cycles and $\nu$ the black.
Then the total number of graphs fixed by such a permutation with a specified decomposition is
\begin{equation*}
  \fix \pbrac{\mu, \nu} = \prod_{\substack{i \in \mu \\ j \in \nu}} 2^{\gcd \pbrac{i, j}}
\end{equation*}
where the product is over the elements of $\mu$ and $\lambda$ taken as multisets.
However, since $\mu$ and $\nu$ represent white and black cycles respectively, it is important to distinguish \emph{which} cycles of $\lambda$ are taken into each.
The $\lambda_{i}$ $i$-cycles of $\lambda$ can be distributed into $\mu$ and $\nu$ in $\binom{\lambda_{i}}{\mu_{i}} = \lambda_{i}! / \pbrac{\mu_{i}! \nu_{i}!}$ ways, so in total there are $\prod_{i} \lambda_{i}! / \pbrac{\mu_{i}! \nu_{i}!} = z_{\lambda} / \pbrac{z_{\mu} z_{\nu}}$ decompositions.
Thus,
\begin{equation*}
  \fix \pbrac{\lambda} = \frac{z_{\lambda}}{z_{\mu} z_{\nu}} \fix \pbrac{\mu, \nu} = \sum_{\mu \cup \nu = \lambda} \frac{z_{\lambda}}{z_{\mu} z_{\nu}} \prod_{\substack{i \in \mu \\ j \in \nu}} 2^{\gcd \pbrac{i, j}}.
\end{equation*}
Therefore we conclude:
\begin{theorem}
  \begin{equation}
    \label{eq:ecibc}
    \gcielt{\symgp{2}}{\specname{BC}}{e} = \sum_{n > 0} \sum_{\substack{\mu, \nu \\ \mu \cup \nu \vdash n}} \frac{p_{\mu \cup \nu}}{z_{\mu} z_{\nu}} \prod_{i, j} 2^{\gcd \pbrac{\mu_{i}, \nu_{j}}}
  \end{equation}
\end{theorem}

Explicit formulas for the generating function for unlabeled bicolored graphs were obtained in \cite{har:bicolored} using conventional P\'{o}lya-theoretic methods.
Conceptually, this enumeration in fact largely mirrors our own.
Harary uses the algebra of the classical cycle index of the `line group\footnote{The \emph{line group} of a graph is the group of permutations of edges induced by permutations of vertices.}' of the complete bicolored graph of which any given bicolored graph is a spanning subgraph.
He then enumerates orbits of edges under these groups using the P\'{o}lya enumeration theorem.
This is clearly analogous to our procedure, which enumerates the orbits of edges under each specific permutation of vertices.

\subsection{Calculating $\gcielt{\symgp{2}}{\specname{BC}}{\tau}$}\label{ss:tcibc}
Recall that the nontrivial element of $\tau \in \symgp{2}$ acts on bicolored graphs by reversing all colors.

We again consider the cycles in the vertex set $\sbrac{n}$ induced by a permutation $\pi \in \symgp{n}$ and use the partition $\lambda$ corresponding to the cycle type of $\pi$ for bookkeeping.
We then wish to count bicolored graphs on $\sbrac{n}$ for which $\tau \cdot \pi$ is an automorphism, which is to say that $\pi$ itself is a color-\emph{reversing} automorphism.
Once again, the number of bicolored graphs for which $\pi$ is a color-reversing automorphism is in fact dependent only on the cycle type $\lambda$.
Each cycle of vertices must be color-alternating and hence of even length, so our partition $\lambda$ must have only even parts.
Once this condition is satisfied, edges may be drawn either within a single cycle or between two cycles, and as before if we draw in any edge we must draw in its entire orbit under $\pi$ (since $\pi$ is to be an automorphism of the underlying graph).
Moreover, all graphs for which $\pi$ is a color-reversing automorphism and with a fixed coloring may be constructed in this way, so it suffices to count such edge orbits and then consider how colorings may be assigned.

Consider a cycle of length $2n$; we hereafter describe such a cycle as having \emph{semilength} $n$.
There are exactly $n^{2}$ possible white-black edges in such a cycle.
If $n$ is odd, diametrically opposed vertices have opposite colors, so we can have an edge of length $l = n$ (in the sense of connecting two vertices which are $l$ steps apart in the cycle), and in such a case the orbit length is exactly $n$ and there is exactly one orbit.
See \cref{fig:exbctincycd} for an example of this case.
However, if $n$ is odd but $l \neq n$, the orbit length is $2n$, so the number of such orbits is $\frac{n^{2} - n}{2n}$.
Hence, the total number of orbits for $n$ odd is $\frac{n^2 + n}{2n} = \ceil{\frac{n}{2}}$.
Similarly, if $n$ is even, all orbits are of length $2n$, so the total number of orbits is $\frac{n^{2}}{2n} = \frac{n}{2} = \ceil{\frac{n}{2}}$ also.
See \cref{fig:exbctincyce} for an example of each of these cases.

\begin{figure}[htb]
  \centering
  \subfloat[A diameter $d$ ($l = 3$)]{\makebox[.45\textwidth]{
      \label{fig:exbctincycd}
      \begin{tikzpicture}
        % \SetGraphUnit{2}
        \GraphInit[vstyle=Hasse]
        
        \SetUpEdge[style=cycedge]
        \grCycle[RA=2,prefix=a]{6}
        \AddVertexColor{black!20}{a0,a2,a4}

        \SetUpEdge[style=dashed]
        \EdgeInGraphMod{a}{6}{3}{0}

        \SetUpEdge[style=solid]
        \Edge[label={$d$},style={pos=.25}](a0)(a3)
      \end{tikzpicture}
    }}
  \hfill
  \subfloat[A non-diameter $e$ ($l = 1)$]{\makebox[.45\textwidth]{
      \label{fig:exbctincyce}
      \begin{tikzpicture}
        % \SetGraphUnit{2}
        \GraphInit[vstyle=Hasse]
        \SetUpEdge[style=cycedge]
        \grCycle[RA=2,prefix=b]{6}
        \AddVertexColor{black!20}{b0,b2,b4}

        \SetUpEdge[style=dashed]
        \EdgeInGraphMod{b}{6}{1}{0}

        \SetUpEdge[style=solid]
        \Edge[label={$e$}](b0)(b1)
      \end{tikzpicture}
    }}
  \caption[Two example edge-orbits in a color-reversing automorphism]{Both types of intra-cycle edges and their orbits on a typical color-alternating $6$-cycle}
  \label{fig:exbctincyc}
\end{figure}
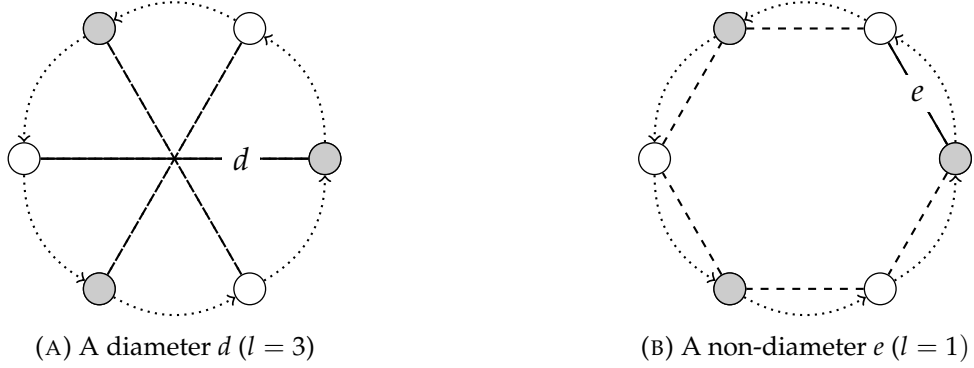

Now consider an edge to be drawn between two cycles of semilengths $m$ and $n$.
The total number of possible white-black edges is $2mn$, each of which has an orbit length of $\lcm \pbrac{2m, 2n} = 2 \lcm \pbrac{m, n}$.
Hence, the total number of orbits is $2mn / \pbrac{2 \lcm \pbrac{m, n}} = \gcd \pbrac{m, n}$.

\begin{figure}[htb]
  \centering
  \begin{tikzpicture}
    \GraphInit[vstyle=Hasse]
    
    \begin{scope}[xshift=-4cm,rotate=30]
      \SetUpEdge[style=cycedge]
      \grCycle[RA=2,prefix=a]{4}
      \AddVertexColor{black!20}{a0,a2}
    \end{scope}

    \begin{scope}[xshift=4cm,rotate=90]
      \SetUpEdge[style=cycedge]
      \grCycle[RA=2,prefix=b]{2}
      \AddVertexColor{black!20}{b0}
    \end{scope}

    \SetUpEdge[style=dashed]
    \EdgeDoubleMod{a}{4}{1}{1}{b}{2}{0}{1}{4}

    \SetUpEdge[style=solid]
    \Edge[label={$e$}](a0)(b1)
  \end{tikzpicture}
  \caption[Another example edge-orbit of a color-reversing automorphism]{An edge $e$ and its orbit between color-alternating cycles of semilengths $2$ and $1$}
  \label{fig:exbctbtwcyc}
\end{figure}
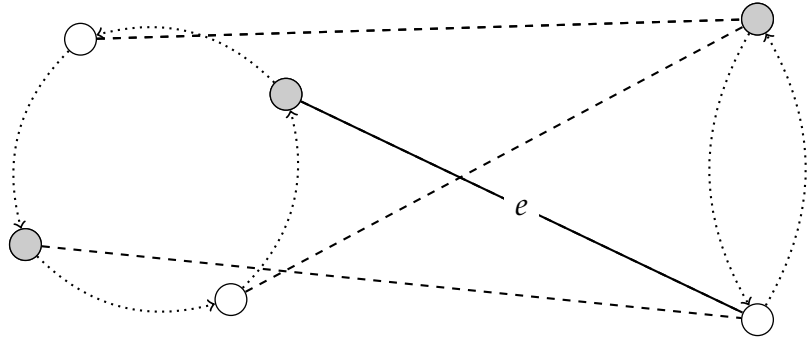

All together, then, the number of orbits for a fixed coloring of a permutation of cycle type $2 \lambda$ (denoting the partition obtained by doubling every part of $\lambda$) is $\sum_{i} \ceil{\frac{\lambda_{i}}{2}} + \sum_{i < j} \gcd \pbrac{\lambda_{i}, \lambda_{j}}$.
All valid bicolored graphs for a fixed coloring for which $\pi$ is a color-preserving automorphism may be obtained uniquely by making some choice of a subset of this collection of orbits, just as in \cref{ss:ecibc}.
Thus, the total number of possible graphs for a given vertex coloring is
\begin{equation*}
  \prod_{i} 2^{\ceil{\frac{\lambda_{i}}{2}}} \prod_{i < j} 2^{\gcd \pbrac{\lambda_{i}, \lambda_{j}}},
\end{equation*}
which we note is independent of the choice of coloring.
For a partition $2\lambda$ with $l \pbrac{\lambda}$ cycles, there are then $2^{l \pbrac{\lambda}}$ colorings compatible with our requirement that each cycle is color-alternating, which we multiply by the previous to obtain the total number of graphs for all permutations $\pi$ with cycle type $2 \lambda$.
Therefore we conclude:
\begin{theorem}
  \begin{equation}
    \label{eq:tcibc}
    \gcielt{\symgp{2}}{\specname{BC}}{\tau} = \sum_{\substack{n > 0 \\ \text{$n$ even}}} \sum_{\lambda \vdash \frac{n}{2}} 2^{l \pbrac{\lambda}} \frac{p_{2 \lambda}}{z_{2 \lambda}} \prod_{i} 2^{\ceil{\frac{\lambda_{i}}{2}}} \prod_{i < j} 2^{\gcd \pbrac{\lambda_{i}, \lambda_{j}}}
  \end{equation}
\end{theorem}

\section{Connected bicolored graphs}\label{s:cbc}
As noted in the introduction of this section, we may pass from bicolored to bipartite graphs by taking a quotient under the color-reversing action of $\symgp{2}$ only in the connected case.
Thus, we must pass from the species $\specname{BC}$ to the species $\specname{CBC}$ of connected bicolored graphs to continue.
It is a standard principle of graph enumeration that a graph may be decomposed uniquely into (and thus species-theoretically identified with) the set of its connected components.
We must, of course, require that the component structures are nonempty to ensure that the construction is well-defined, as discussed in \cref{s:specalg}.
This same relationship holds in the case of bicolored graphs.
Thus, the species $\specname{BC}$ of nonempty bicolored graphs is the composition of the species $\specname{CBC}$ of nonempty connected bicolored graphs into the species $\specname{E}^{+} = \specname{E} - 1$ of nonempty sets:
\begin{equation} \specname{BC} = \specname{E}^{+} \circ \specname{CBC} \label{eq:bcdecomp} \end{equation}

Reversing the colors of a bicolored graph is done simply by reversing the colors of each of its connected components independently; thus, once we trivially extend the species $\specname{E}^{+}$ to an $\symgp{2}$-species by applying the trivial action, \cref{eq:bcdecomp} holds as an identity of $\symgp{2}$-species for the color-reversing $\symgp{2}$-action described previously.

To use the decomposition in \cref{eq:bcdecomp} to derive the $\symgp{2}$-cycle index for $\specname{CBC}$, we must invert the $\symgp{2}$-species composition into $\specname{E}^{+}$.
In the context of the theory of virtual species, this is possible; we write $\con := \pbrac{\specname{E} - 1}^{\abrac{-1}}$ to denote this virtual species.
We can derive from \cite[\S 2.5, eq.~(58c)]{bll:species} that its cycle index is
\begin{equation}
  \label{eq:zgamma}
  \ci{\con} = \sum_{k \geq 1} \frac{\mu \pbrac{k}}{k} \log \pbrac{1 + p_{k}}
\end{equation}
where $\mu$ is the M\"{o}bius function.
We can then rewrite \cref{eq:bcdecomp} as
\[\specname{CBC} = \con \circ \specname{BC}\]
It then follows immediately from \cref{thm:gspeccompci} that
\begin{theorem}
  \begin{equation} \gci{\symgp{2}}{\specname{CBC}} = \ci{\con} \circ \gci{\symgp{2}}{\specname{BC}} \label{eq:zcbcdecomp} \end{equation}
\end{theorem}

\section{Bipartite graphs}\label{s:bp}
As we previously observed, connected bipartite graphs are naturally identified with orbits of connected bicolored graphs under the color-reversing action of $\symgp{2}$.
Thus,
\begin{equation*}
  \specname{CBP} = \faktor{\specname{CBC}}{\symgp{2}}.
\end{equation*}
By application of \cref{thm:qsci}, we can then directly compute the cycle index of $\specname{CBP}$ in terms of previous results:
\begin{theorem}
  \begin{equation}
    \ci{\specname{CBP}} = \qgci{\symgp{2}}{\specname{CBC}} = \frac{1}{2} \pbrac{\gcielt{\symgp{2}}{\specname{CBC}}{e} + \gcielt{\symgp{2}}{\specname{CBC}}{\tau}}.
  \end{equation}
\end{theorem}

Finally, to reach a result for the general bipartite case, we return to the graph-theoretic composition relationship previously considered in \cref{s:cbc}:
\begin{equation*}
  \specname{BP} = \specname{E} \circ \specname{CBP}.
\end{equation*}

This time, we need not invert the composition, so the cycle-index calculation is simple:
\begin{theorem}
  \begin{equation}
    \ci{\specname{BP}} = \ci{\specname{E}} \circ \ci{\specname{CBP}}.
  \end{equation}
\end{theorem}

A generating function for labeled bipartite graphs was obtained first in \cite{harprins:bipartite} and later in \cite{han:bipartite}; the latter uses P\'{o}lya-theoretic methods to calculate the cycle index of what in modern terminology would be the species of edge-labeled complete bipartite graphs.

\section{Nonseparable graphs}\label{s:nbp}
We now turn our attention to the notions of block decomposition and nonseparable graphs.
A graph is said to be \emph{nonseparable} if it is vertex-$2$-connected (that is, if there exists no vertex whose removal disconnects the graph); every connected graph then has a canonical `decomposition'\footnote{Note that this decomposition does not actually partition the vertices, since many blocks may share a single cut-point, a detail which significantly complicates but does not entirely preclude species-theoretic analysis.} into maximal nonseparable subgraphs, often shortened to \emph{blocks}.
In the spirit of our previous notation, we we will denote by $\specname{NBP}$ the species of nonseparable bipartite graphs, our object of study.

The basic principles of block enumeration in terms of automorphisms and cycle indices of permutation groups were first identified and exploited in \cite{rob:nonsep}.
In \cite[\S 4.2]{bll:species}, a theory relating a specified species $B$ of nonseparable graphs to the species $C_{B}$ of connected graphs whose blocks are in $B$ is developed using similar principles.
It is apparent that the class of nonseparable bipartite graphs is itself exactly the class of blocks that occur in block decompositions of connected bipartite graphs; hence, we apply that theory here to study the species $\specname{NBP}$.
From \cite[eq.~4.2.27]{bll:species} we obtain
\begin{theorem}
  \begin{subequations}
    \label{eq:nbpexp}
    \begin{equation}
      \label{eq:nbpexpmain}
      \specname{NBP} = \specname{CBP} \pbrac{\specname{CBP}^{\bullet \abrac{-1}}} + X \cdot \deriv{\specname{NBP}} - X,
    \end{equation}
    where by \cite[4.2.26(a)]{bll:species} we have
    \begin{equation}
      \label{eq:nbpexpsub}
      \deriv{\specname{NBP}} = \con \pbrac{\frac{X}{\specname{CBP}^{\bullet \abrac{-1}}}}.
    \end{equation}
  \end{subequations}

\end{theorem}
We have already calculated the cycle index for the species $\specname{CBP}$, so the calculation of the cycle index of $\specname{NBP}$ is now simply a matter of algebraic expansion.

A generating function for labeled bipartite blocks was given in \cite{harrob:bipblocks}, where their analogue of \cref{eq:nbpexp} for the labeled exponential generating function for blocks comes from \cite{forduhl:combprob1}.
However, we could locate no corresponding unlabeled enumeration in the literature.
The numbers of labeled and unlabeled nonseparable bipartite graphs for $n \leq 10$ as calculated using our method are given in \cref{tab:bpblocks}.

\chapter{The species of $k$-trees}\label{c:ktrees}
\section{Introduction}\label{s:intro}
\subsection{$k$-trees}\label{ss:ktrees}
Trees and their generalizations have played an important role in the literature of combinatorial graph theory throughout its history.
The multi-dimensional generalization to so-called `$k$-trees' has proved to be particularly fertile ground for both research problems and applications.

The class $\kt{k}$ of $k$-trees (for $k \in \ringname{N}$) may be defined recursively:
\begin{definition}
  \label{def:ktree}
  The complete graph on $k$ vertices ($K_{k}$) is a $k$-tree, and any graph formed by adding a single vertex to a $k$-tree and connecting that vertex by edges to some existing $k$-clique (that is, induced $k$-complete subgraph) of that $k$-tree is a $k$-tree.
\end{definition}

The graph-theoretic notion of $k$-trees was first introduced in 1968 in \cite{harpalm:acycsimp}; vertex-labeled $k$-trees were quickly enumerated in the following year in both \cite{moon:lktrees} and \cite{beinpipp:lktrees}.
The special case $k=2$ has been especially thoroughly studied; enumerations are available in the literature for edge- and triangle-labeled $2$-trees in \cite{palm:l2trees}, for plane $2$-trees in \cite{palmread:p2trees}, and for unlabeled $2$-trees in \cite{harpalm:acycsimp} and \cite{harpalm:graphenum}.
In 2001, the theory of species was brought to bear on $2$-trees in \cite{gessel:spec2trees}, resulting in more explicit formulas for the enumeration of unlabeled $2$-trees.
An extensive literature on other properties of $k$-trees and their applications has also emerged; Beineke and Pippert claim in \cite{beinpipp:multidim} that ``[t]here are now over 100 papers on various aspects of $k$-trees''.
However, no general enumeration of unlabeled $k$-trees appears in the literature to date.

To begin, we establish two definitions for substructures of $k$-trees which we will use extensively in our analysis.
\begin{definition}
  \label{def:hedfront}
  A \emph{hedron} of a $k$-tree is a $\pbrac{k+1}$-clique and a \emph{front} is a $k$-clique.
\end{definition}
We will frequently describe $k$-trees as assemblages of hedra attached along their fronts rather than using explicit graph-theoretic descriptions in terms of edges and vertices, keeping in mind that the structure of interest is graph-theoretic and not geometric.
The recursive addition of a single vertex and its connection by edges to an existing $k$-clique in \cref{def:ktree} is then interpreted as the attachment of a hedron to an existing one along some front, identifying the $k$ vertices they have in common.
The analogy to the recursive definition of conventional trees is clear, and in fact the class $\mathfrak{a}$ of trees may be recovered by setting $k = 1$.
For higher $k$, the structures formed are still distinctively tree-like; for example, $2$-trees are formed by gluing triangles together along their edges without forming loops of triangles (see \cref{fig:ex2tree}), while $3$-trees are formed by gluing tetrahedra together along their triangular faces without forming loops of tetrahedra.

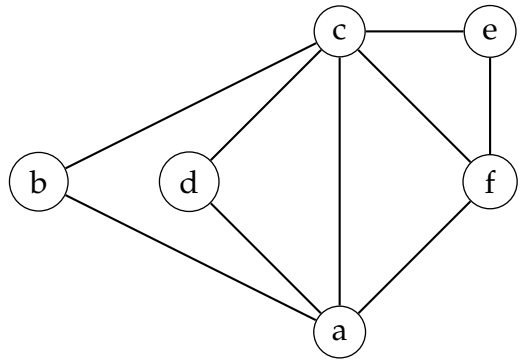
\begin{figure}[htb]
  \centering
  \begin{tikzpicture}
    \SetGraphUnit{2}
    \GraphInit[vstyle=normal]
    \Vertex{a}

    \NOWE(a){d}
    \Edge(a)(d)

    \NOEA(d){c}
    \Edge(a)(c)
    \Edge(d)(c)
    
    \WE(d){b}
    \Edge(a)(b)
    \Edge(c)(b)
    
    \NOEA(a){f}
    \Edge(a)(f)
    \Edge(c)(f)
    
    \NO(f){e}
    \Edge(f)(e)
    \Edge(c)(e)
  \end{tikzpicture}
  \caption{A (vertex-labeled) $2$-tree}
  \label{fig:ex2tree}
\end{figure}

In graph-theoretic contexts, it is conventional to label graphs on their vertices and possibly their edges.
However, for our purposes, it will be more convenient to label hedra and fronts.
Throughout, we will treat the species $\kt{k}$ of $k$-trees as a two-sort species, with $X$-labels on the hedra and $Y$-labels on their fronts; in diagrams, we will generally use capital letters for the hedron-labels and positive integers for the front-labels (see \cref{fig:exlab2tree}).

\begin{figure}[htb]
  \centering
  \begin{tikzpicture}
    \SetGraphUnit{3}
    \GraphInit[vstyle=Hasse]
    \SetUpEdge[labelstyle={draw}]

    \Vertex{a}

    \NOWE(a){d}
    \Edge[label=2](a)(d)

    \NOEA(d){c}
    \Edge[label=4](a)(c)
    \Edge[label=5](d)(c)

    \node at (barycentric cs:a=1,d=1,c=1) {B};
    
    \WE(d){b}
    \Edge[label=1](a)(b)
    \Edge[label=6](c)(b)

    \node at (barycentric cs:b=1,d=1) {D};
    
    \NOEA(a){f}
    \Edge[label=9](a)(f)
    \Edge[label=3](c)(f)

    \node at (barycentric cs:a=1,c=1,f=1) {C};
    
    \NO(f){e}
    \Edge[label=7](f)(e)
    \Edge[label=8](c)(e)

    \node at (barycentric cs:f=1,c=1,e=1) {A};

  \end{tikzpicture}
  \caption{A (hedron-and-front--labeled) $2$-tree}
  \label{fig:exlab2tree}
\end{figure}
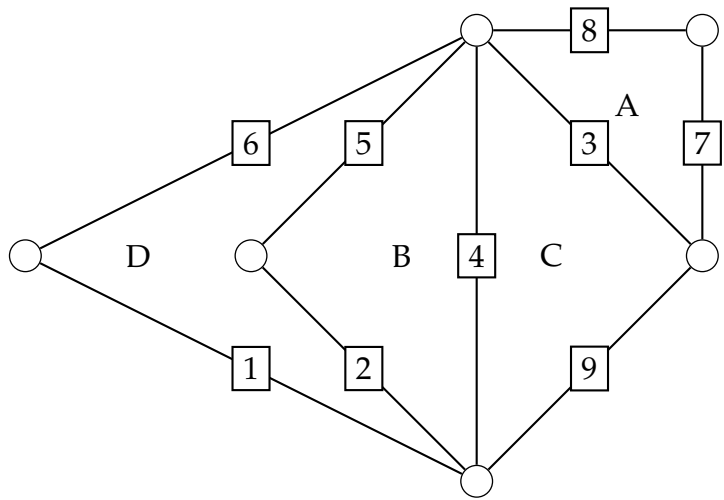

\section{The dissymmetry theorem for $k$-trees}\label{s:dissymk}
Studies of tree-like structures---especially those explicitly informed by the theory of species---often feature decompositions based on \emph{dissymmetry}, which allow enumerations of unrooted structures to be recharacterized in terms of rooted structures.
For example, as seen in \cite[\S 4.1]{bll:species}, the species $\mathfrak{a}$ of trees and $\specname{A} = \pointed{\mathfrak{a}}$ of rooted trees are related by the equation
\begin{equation*}
  \specname{A} + \specname{E}_{2} \pbrac{\specname{A}} = \mathfrak{a} + \specname{A}^{2}
\end{equation*}
where the proof hinges on a recursive structural decomposition of trees.
In this case, the species $\specname{A}$ is relatively easy to characterize explicitly, so this equation serves to characterize the species $\mathfrak{a}$, which would be difficult to do directly.

A similar theorem holds for $k$-trees.
\begin{theorem}
  \label{thm:dissymk}
  The species $\ktx{k}$ and $\kty{k}$ of $k$-trees rooted at hedra and fronts respectively, $\ktxy{k}$ of $k$-trees rooted at a hedron with a designated front, and $\kt{k}$ of unrooted $k$-trees are related by the equation
  \begin{equation}
    \label{eq:dissymk}
    \ktx{k} + \kty{k} = \kt{k} + \ktxy{k}
  \end{equation}
  as an isomorphism of species.
\end{theorem}

\begin{proof}
  We give a bijective, natural map from $\pbrac{\ktx{k} + \kty{k}}$-structures on the left side to $\pbrac{\kt{k} + \ktxy{k}}$-structures on the right side.
  Define a \emph{$k$-path} in a $k$-tree to be a non-self-intersecting sequence of consecutively adjacent hedra and fronts, and define the \emph{length} of a $k$-path to be the total number of hedra and fronts along it.
  Note that the ends of every maximal $k$-path in a $k$-tree are fronts.
  It is easily verified, as in \cite{kob:ktlogspace}, that every $k$-tree has a unique \emph{center} clique (either a hedron or a front) which is the midpoint of every longest $k$-path (or, equivalently, has the greatest $k$-eccentricity, defined appropriately).
  
  An $\pbrac{\ktx{k} + \kty{k}}$-structure on the left-hand side of the equation is a $k$-tree $T$ rooted at some clique $c$, which is either a hedron or a front.
  Suppose that $c$ is the center of $T$.
  We then map $T$ to its unrooted equivalent in $\kt{k}$ on the right-hand side.
  This map is a natural bijection from its preimage, the set of $k$-trees rooted at their centers, to $\kt{k}$, the set of unrooted $k$-trees.

  Now suppose that the root clique $c$ of the $k$-tree $T$ is \emph{not} the center, which we denote $C$.
  Identify the clique $c'$ which is adjacent to $c$ along the $k$-path from $c$ to $C$.
  We then map the $k$-tree $T$ rooted at the clique $c$ to the same tree $T$ rooted at \emph{both} $c$ and its neighbor $c'$.
  This map is also a natural bijection, in this case from the set of $k$-trees rooted at vertices which are \emph{not} their centers to the set $\ktxy{k}$ of $k$-trees rooted at an adjacent hedron-front pair.

  The combination of these two maps then gives the desired isomorphism of species in \cref{eq:dissymk}.
\end{proof}

In general we will reformulate the dissymmetry theorem as follows:
\begin{corollary}
  \label{cor:dissymkreform}
  For the various forms of the species $\kt{k}$ as above, we have
  \begin{equation}
    \label{eq:dissymkreform}
    \kt{k} = \ktx{k} + \kty{k} - \ktxy{k}.
  \end{equation}
  as an isomorphism of ordinary species.
\end{corollary}

This species subtraction is well-defined in the sense of \cref{def:specdif}, since the species $\ktxy{k}$ embeds in the species $\ktx{k} + \kty{k}$ by the centering map described in the proof of \cref{thm:dissymk}.
Essentially, \cref{eq:dissymkreform} identifies each unrooted $k$-tree with itself rooted at its center simplex.

\Cref{thm:dissymk} and the consequent \cref{eq:dissymkreform} allow us to reframe enumerative questions about generic $k$-trees in terms of questions about $k$-trees rooted in various ways.
However, the rich internal symmetries of large cliques obstruct direct analysis of these rooted structures.
We need to break these symmetries to proceed.

\section{Coherently-oriented $k$-trees}
\subsection{Symmetry-breaking}\label{ss:symbreak}
In the case of the species $\specname{A} = \pointed{\kt{1}}$ of rooted trees, we may obtain a simple recursive functional equation \cite[\S 1, eq.~(9)]{bll:species}:
\begin{equation}
  \label{eq:rtrees}
  \specname{A} = X \cdot \specname{E} \pbrac{\specname{A}}.
\end{equation}
This completely characterizes the combinatorial structure of the class of trees.

However, in the more general case of $k$-trees, no such simple relationship obtains; attached to a given hedron is a collection of sets of hedra (one such set per front), but simply specifying which fronts to attach to which does not fully specify the attachings, and the structure of that collection of sets is complex.
We will break this symmetry by adding additional structure which we can later remove using the theory of quotient species.

\begin{definition}
  \label{def:mirrorfronts}
  Let $h_{1}$ and $h_{2}$ be two hedra joined at a front $f$, hereafter said to be \emph{adjacent}.
  Each other front of one of the hedra shares $k-1$ vertices with $f$; we say that two fronts $f_{1}$ of $h_{1}$ and $f_{2}$ of $h_{2}$ are \emph{mirror with respect to $f$} if these shared vertices are the same, or equivalently if $f_{1} \cap f = f_{2} \cap f$.
\end{definition}

\begin{observation}
  \label{obs:mirrorfronts}
  Let $T$ be a coherently-oriented $k$-tree with two hedra $h_{1}$ and $h_{2}$ joined at a front $f$.
  Then there is exactly one front of $h_{2}$ mirror to each front of $h_{1}$ with respect to their shared front $f$.
\end{observation}

\begin{comment} Don't think we need this after all...
  \begin{definition}\label{def:cycord}
    For a set $A$, define a \emph{cyclic order of $A$} to be a labeling of the cyclic digraph $\overrightarrow{C}_{\abs{A}}$ by $A$ and let $\cyc A$ be the set of such linear orders. Let $\lin A$ be the set of linear orders on $A$. Let $\psi_{A}: \lin A \to \cyc A$ (hereafter denoted simply $\psi$ when the set is unambiguous) send each linear order $\ell$ to the cyclic order obtained by decorating $\overrightarrow{C}_{\abs{A}}$ with $\ell$ in order. (Note that this map is $\abs{A}$-to-one.) Then a \emph{linearization} of a given cyclic order $c \in \cyc A$ is an element of $\psi^{-1} \pbrac{c}$.
  \end{definition}
\end{comment}

\begin{definition}
  \label{def:coktree}
  Define an \emph{orientation} of a hedron to be a cyclic ordering of the set of its fronts and an \emph{orientation} of a $k$-tree to be a choice of orientation for each of its hedra.
  If two oriented hedra share a front, their orientations are \emph{compatible} if they correspond under the mirror bijection.
  Then an orientation of a $k$-tree is \emph{coherent} if every pair of adjacent hedra is compatibly-oriented.
\end{definition}
See \cref{fig:exco2tree} for an example.
Note that every $k$-tree admits many coherent orientations---any one hedron of the $k$-tree may be oriented freely, and a unique orientation of the whole $k$-tree will result from each choice of such an orientation of one hedron.
We will denote by $\ktco{k}$ the species of coherently-oriented $k$-trees.

By shifting from the general $k$-tree setting to that of coherently-oriented $k$-trees, we break the symmetry described above.
If we can now establish a group action on $\ktco{k}$ whose orbits are generic $k$-trees we can use the theory of quotient species to extract the generic species $\kt{k}$.
First, however, we describe an encoding procedure which will make future work more convenient.

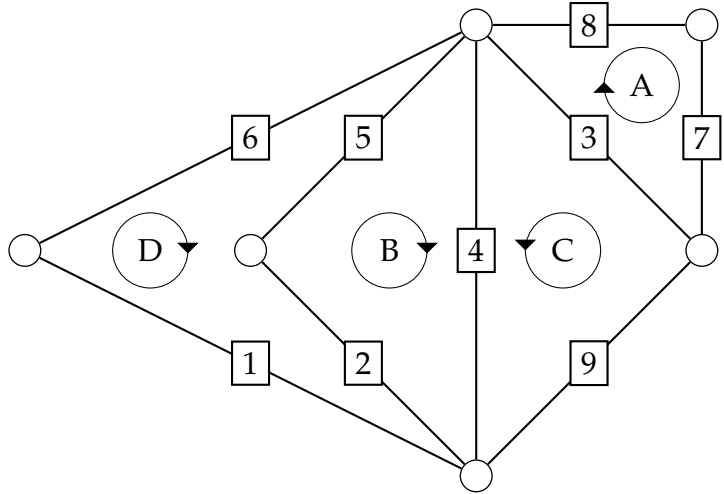
\begin{figure}[htb]
  \centering
  \begin{tikzpicture}
    \SetGraphUnit{3}
    \GraphInit[vstyle=Hasse]
    \SetUpEdge[labelstyle={draw}]

    \Vertex{a}

    \NOWE(a){d}
    \Edge[label=2](a)(d)

    \NOEA(d){c}
    \Edge[label=4](a)(c)
    \Edge[label=5](d)(c)

    \coordinate (B) at (barycentric cs:a=1,d=1.25,c=1);
    \node at (B) {B};
    \cyccwr[.5cm]{(B)};
    
    \WE(d){b}
    \Edge[label=1](a)(b)
    \Edge[label=6](c)(b)

    \coordinate (D) at (barycentric cs:b=1,d=1.25);
    \node at (D) {D};
    \cyccwr[.5cm]{(D)};
    
    \NOEA(a){f}
    \Edge[label=9](a)(f)
    \Edge[label=3](c)(f)

    \coordinate (C) at (barycentric cs:a=1,c=1,f=1.25);
    \node at (C) {C};
    \cycccwl[.5cm]{(C)};
    
    \NO(f){e}
    \Edge[label=7](f)(e)
    \Edge[label=8](c)(e)

    \coordinate (A) at (barycentric cs:f=1,c=1,e=1.75);
    \node at (A) {A};
    \cyccwl[.5cm]{(A)};
    
  \end{tikzpicture}
  \caption{A coherently-oriented $2$-tree}
  \label{fig:exco2tree}
\end{figure}

\subsection{Bicolored tree encoding}\label{ss:bctree}
Although $k$-trees are graphs (and hence made up simply of edges and vertices), their structure is more conveniently described in terms of their simplicial structure of hedra and fronts.
Indeed, if each hedron has an orientation of its faces and we choose in advance which hedra to attach to which by what fronts, the requirement that the resulting $k$-tree be coherently oriented is strong enough to characterize the attaching completely.
We thus pass from coherently-oriented $k$-trees to a surrogate structure which exposes the salient features of this attaching structure more clearly---structured bicolored trees in the spirit of the $R, S$-enriched bicolored trees of \cite[\S 3.2]{bll:species}.

A $\pbrac{\specname{C}_{k+1}, \specname{E}}$-enriched bicolored tree is a bicolored tree each black vertex of which carries a $\specname{C}_{k+1}$-structure (that is, a cyclic ordering on $k+1$ elements) on its white neighbors.
(The $\specname{E}$-structure on the black neighbors of each white vertex is already implicit in the bicolored tree itself.)
For later convenience, we will sometimes call such objects \emph{$k$-coding trees}, and we will denote by $\ct{k}$ the species of such $k$-coding trees.

We now define a map $\beta: \ktco{k} \sbrac{n} \to \ct{k} \sbrac{n}$.
For a given coherently-oriented $k$-tree $T$ with $n$ hedra:
\begin{itemize}
  \item For every hedron of $T$ construct a black vertex and for every front a white vertex, assigning labels appropriately.
  \item For every black-white vertex pair, construct a connecting edge if the white vertex represents a front of the hedron represented by the black vertex.
  \item Finally, enrich the collection of neighbors of each black vertex with a $\specname{C}_{k+1}$-structure inherited directly from the orientation of the $k$-tree $T$.
\end{itemize}
The resulting object $\beta \pbrac{T}$ is clearly a $k$-coding tree with $n$ black vertices.

We can recover a $T$ from $\beta \pbrac{T}$ by following the reverse procedure.
For an example, see \cref{fig:exbctree}, which shows the $2$-coding tree associated to the coherently-oriented $2$-tree of \cref{fig:exco2tree}.
Note that, for clarity, we have rendered the black vertices (corresponding to hedra) with squares.

\begin{figure}[htb]
  \centering
  \begin{tikzpicture}
    \SetGraphUnit{1.5}
    \GraphInit[vstyle=normal]

    \Vertex[style=ynode]{4}

    \SOWE[style=xnode](4){B}
    \cyccwr{(B)};
    \Edge(4)(B)

    \WE[style=ynode](B){5}
    \Edge(B)(5)

    \SOEA[style=ynode](B){2}
    \Edge(B)(2)

    \NOWE[style=xnode](4){D}
    \cyccwr{(D)};
    \Edge(4)(D)
    
    \WE[style=ynode](D){1}
    \Edge(D)(1)

    \NOEA[style=ynode](D){6}
    \Edge(D)(6)

    \EA[style=xnode](4){C}
    \cycccwr{(C)};
    \Edge(4)(C)

    \SOEA[style=ynode](C){9}
    \Edge(C)(9)

    \NOEA[style=ynode](C){3}
    \Edge(C)(3)

    \NO[style=xnode](3){A}
    \cyccwl{(A)};
    \Edge(3)(A)

    \NOWE[style=ynode](A){8}
    \Edge(A)(8)

    \NOEA[style=ynode](A){7}
    \Edge(A)(7)
    
  \end{tikzpicture}
  \caption{A $\pbrac{\specname{C}_{k+1}, \specname{E}}$-enriched bicolored tree encoding a coherently-oriented $2$-tree}
  \label{fig:exbctree}
\end{figure}

\begin{theorem}\label{thm:bctreeenc}
  The map $\beta$ induces an isomorphism of species $\ktco{k} \simeq \ct{k}$.
\end{theorem}

\begin{proof}
  It is clear that $\beta$ sends each coherently-oriented $k$-tree to a unique $k$-coding tree, and that this map commutes with permutations on the label sets (and thus is categorically natural).
  To show that $\beta$ induces a species isomorphism, then, we need only show that $\beta$ is a surjection onto $\ct{k} \sbrac{n}$ for each $n$.
  Throughout, we will say `$F$ and $G$ have contact of order $n$' when the restrictions $F_{\leq n}$ and $G_{\leq n}$ of the species $F$ and $G$ to label sets of cardinality at most $n$ are naturally isomorphic.
  
  First, we note that there are exactly $k!$ coherently-oriented $k$-trees with one hedron---one for each cyclic ordering of the $k+1$ front labels.
  There are also $k!$ coding trees with one black vertex, and the encoding $\beta$ is clearly a natural bijection between these two sets.
  Thus, the species $\ktco{k}$ of coherently-oriented $k$-trees and $\ct{k}$ of $k$-coding trees have contact of order $1$. 

  Now, by way of induction, suppose $\ktco{k}$ and $\ct{k}$ have contact of order $n \geq 1$.
  Let $C$ be a $k$-coding tree with $n+1$ black vertices.
  Then let $C_{1}$ and $C_{2}$ be two distinct sub-$k$-coding trees of $C$, each obtained from $C$ by removing one black node which has only one white neighbor which is not a leaf.
  Then, by hypothesis, there exist coherently-oriented $k$-trees $T_{1}$ and $T_{2}$ with $n$ hedra such that $\beta \pbrac{T_{1}} = C_{1}$ and $\beta \pbrac{T_{2}} = C_{2}$.
  Moreover, $\beta \pbrac{T_{1} \cap T_{2}} = \beta \pbrac{T_{1}} \cap \beta \pbrac{T_{2}}$, and this $k$-coding tree has $n-1$ black vertices, so $T_{1} \cap T_{2}$ has $n-1$ hedra.
  Thus, $T = T_{1} \cup T_{2}$ is a coherently-oriented $k$-tree with $n+1$ black hedra, and $\beta \pbrac{T} = C$ as desired.
  Thus, $\beta^{-1} \pbrac{\beta \pbrac{T_{1}} \cup \beta \pbrac{T_{2}}} = T_{1} \cup T_{2} = T$, and hence $\ktco{k}$ and $\ct{k}$ have contact of order $n+1$.
\end{proof}

Thus, $\ktco{k}$ and $\ct{k}$ are isomorphic as species; however, $k$-coding trees are much simpler than coherently-oriented $k$-trees as graphs.
Moreover, $k$-coding trees are doubly-enriched bicolored trees as in \cite[\S 3.2]{bll:species}, for which the authors of that text develop a system of functional equations which fully characterizes the cycle index of such a species.
We thus will proceed in the following sections with a study of the species $\ct{k}$, then lift our results to the $k$-tree context.

\subsection{Functional decomposition of $k$-coding trees}\label{ss:codecomp}
With the encoding $\beta: \ktco{k} \to \ct{k}$, we now have direct graph-theoretic access to the attaching structure of coherently-oriented $k$-trees.
We therefore turn our attention to the $k$-coding trees themselves to produce a recursive decomposition.
As with $k$-trees, we will study rooted versions of the species $\ct{k}$ of $k$-coding trees first, then use dissymmetry to apply the results to unrooted enumeration.

Let $\ctx{k}$ denote the species of $k$-coding trees rooted at black vertices, $\cty{k}$ denote the species of $k$-coding trees rooted at white vertices, and $\ctxy{k}$ denote the species of $k$-coding trees rooted at edges (that is, at adjacent black-white pairs).
By construction, a $\ctx{k}$-structure consists of a single $X$-label and a cyclically-ordered $\pbrac{k+1}$-set of $\cty{k}$-structures.
See \cref{fig:ctxconst} for an example of this construction.

\begin{figure}[htb]
  \centering
  \def\kval{4}
  \begin{tikzpicture}
    \pgfmathparse{int(1+\kval)} %int necessary for clean captions later
    \let\numfronts\pgfmathresult

    \pgfmathparse{180/\numfronts}
    \let\childshift\pgfmathresult

    \node [style=xnode] (root) at (0,0) {$X$};   
    \cycccwl{(root)};

    \draw (180/\numfronts:1) node {$\specname{C}_{\numfronts}$};

    \foreach \i in {0, ..., \kval} {
      \pgfmathparse{360*\i/\numfronts}
      \let\theta\pgfmathresult     
      
      \node [style=ynode] (child\i) at (\theta:3) {$Y$};
      \path [draw] (root) -- (child\i);
      \draw (child\i) ++(\theta+90:1) node {$\cty{\kval}$};

      \draw (child\i) ++(\theta:1cm) ++(180+\theta-\childshift:2cm) arc (180+\theta-\childshift:180+\theta+\childshift:2cm);

      \path [draw] (child\i) -- ++(\theta:2) node [rotate=\theta,fill=white] {$\cdots$};
      \path [draw] (child\i) -- ++(\theta+\childshift:2);
      \path [draw] (child\i) -- ++(\theta-\childshift:2);
    }
    
  \end{tikzpicture}
  \caption[An example $X$-rooted $k$-coding tree]{An example $\ctx{\kval}$-structure, rooted at the $X$-vertex.}
  \label{fig:ctxconst}
\end{figure}

Similarly, a $\cty{k}$-structure essentially consists of a single $Y$-label and a (possibly empty) set of $\ctx{k}$-structures, but with some modification.
Every white neighbor of the black root of a $\ctx{k}$-structure is labeled in the construction above, but the white parent of a $\ctx{k}$-structure in this recursive decomposition is already labeled.
Thus, the structure around a black vertex which is a child of a white vertex consists of an $X$ label and a linearly-ordered $k$-set of $\cty{k}$-structures.
Thus, a $\cty{k}$-structure consists of a $Y$-label and a set of pairs of an $X$ label and an $\specname{L}_{k}$-structure of $\cty{k}$-structures.
We note here for conceptual consistency that in fact $\specname{L}_{k} = \deriv{\specname{C}}_{k+1}$ for $\specname{L}$ the species of linear orders and $\specname{C}$ the species of cyclic orders and that $\deriv{\specname{E}} = \specname{E}$ for $\specname{E}$ the species of sets; readers familiar with the $R, S$-enriched bicolored trees of \cite[\S 3.2]{bll:species} will recognize echoes of their decomposition in these facts.

Finally, a $\ctxy{k}$-structure is simply an $X \cdot \specname{L}_{k} \pbracs[big]{\cty{k}}$-structure as described above (corresponding to the black vertex) together with a $\cty{k}$-structure (corresponding to the white vertex).
For reasons that will become clear later, we note that we can incorporate the root white vertex into the linear order by making it last, thus representing a $\ctxy{k}$-structure instead as an $X \cdot \specname{L}_{k+1} \pbracs[big]{\cty{k}}$-structure.
See \cref{fig:ctxyconst} for an example of this construction.

\begin{figure}[htb]
  \centering
  \def\kval{4}
  \begin{tikzpicture}
    \pgfmathparse{int(1+\kval)} %int necessary for clean captions later
    \let\numfronts\pgfmathresult

    \pgfmathparse{180/\numfronts}
    \let\childshift\pgfmathresult

    \node [style=xnode] (root) at (0,0) {$X$};
    \draw[->,postaction={decorate},decoration={markings, mark = at position 1 with {\arrow{triangle 90}}}] (root) ++(2*\childshift:1cm) arc (2*\childshift:360:1cm); %What a mess! At least it works now...

    \draw (root) -- ++(0:3) [ultra thick] node [fill=white] {};

    \draw (180/\numfronts:1) node {$\specname{L}_{\numfronts}$};

    \foreach \i in {0, ..., \kval} {
      \pgfmathparse{360*\i/\numfronts}
      \let\theta\pgfmathresult     
      
      \node [style=ynode] (child\i) at (\theta:3) {$Y$};
      \path [draw] (root) -- (child\i);
      \draw (child\i) ++(\theta+90:1) node {$\cty{\kval}$};

      \draw (child\i) ++(\theta:1cm) ++(180+\theta-\childshift:2cm) arc (180+\theta-\childshift:180+\theta+\childshift:2cm);

      \path [draw] (child\i) -- ++(\theta:2) node [rotate=\theta,fill=white] {$\cdots$};
      \path [draw] (child\i) -- ++(\theta+\childshift:2);
      \path [draw] (child\i) -- ++(\theta-\childshift:2);
    }
    
  \end{tikzpicture}
  \caption[An example $XY$-rooted $k$-coding tree]{An example $\ctxy{\kval}$-structure, rooted at the $X$-vertex and the thick edge adjoining it.}
  \label{fig:ctxyconst}
\end{figure}

The various species of rooted $k$-coding trees are therefore related by a system of functional equations:
\begin{observation}
  \label{obs:funcdecompct}
  For the (ordinary) species $\ctx{k}$ of $X$-rooted $k$-coding trees, $\cty{k}$ of $Y$-rooted $k$-coding trees, and $\ctxy{k}$ of edge-rooted $k$-coding trees, we have the functional relationships
  \begin{subequations}
    \label{eq:ctfunc}
    \begin{align}
      \ctx{k} &= X \cdot \specname{C}_{k+1} \pbracs[big]{\cty{k}} \label{eq:ctxfunc} \\
      \cty{k} &= Y \cdot \specname{E} \pbrac{X \cdot \specname{L}_{k} \pbracs[big]{\cty{k}}} \label{eq:ctyfunc} \\
      \ctxy{k} &= \cty{k} \cdot X \cdot \specname{L}_{k} \pbracs[big]{\cty{k}} = X \cdot \specname{L}_{k+1} \pbracs[big]{\cty{k}} \label{eq:ctxyfunc}
    \end{align}
  \end{subequations}
  as isomorphisms of ordinary species.
\end{observation}

However, a recursive characterization of the various ordinary species of $k$-coding trees is insufficient to characterize the species of $k$-trees itself, since $k$-coding trees represent $k$-trees with coherent orientations.

\section{Generic $k$-trees}\label{s:genkt}
To remove the additional structure of coherent orientation imposed on $k$-trees before their conversion to $k$-coding trees, we now apply the theory of $\Gamma$-species developed in \cref{s:quot}.
In \cite{gessel:spec2trees}, the orientation-reversing action of $\symgp{2}$ on $\cyc_{\sbrac{3}}$ is exploited to study $2$-trees species-theoretically.
We might hope to develop an analogous group action under which general $k$-trees are naturally identified with orbits of coherently-oriented $k$-trees under an action of $\symgp{k}$.
Unfortunately:
\begin{proposition}
  \label{prop:notransac}
  For $k \geq 3$, no transitive action of any group on the set $\cyc_{\sbrac{k+1}}$ of cyclic orders on $\sbrac{k+1}$ commutes with the action of $\symgp{k+1}$ that permutes labels.
\end{proposition}
\begin{proof}
  We represent the elements of $\cyc_{\sbrac{k+1}}$ as cyclic permutations on the alphabet $\sbrac{k+1}$; then the action of $\symgp{k+1}$ that permutes labels is exactly the conjugation action on these permutations.
  Consider an action of a group $G$ on $\cyc_{\sbrac{k+1}}$ that commutes with this conjugation action.
  Then, for any $g \in G$ and any $c \in \cyc_{\sbrac{k+1}}$, we have that
  \begin{equation}
    \label{eq:transaction}
    g \cdot c = g \cdot c c c^{-1} = c \pbrac{g \cdot c} c^{-1}
  \end{equation}
  and so $c$ and $g \cdot c$ commute.
  Thus, $c$ commutes with every element of its orbit under the action of $G$.
  But, for $k \geq 3$, not all elements of $\cyc_{\sbrac{k+1}}$ commute, so the action is not transitive.
\end{proof}

We thus cannot hope to attack the coherent orientations of $k$-trees by acting directly on the cyclic orderings of fronts.
Accordingly, we cannot simply apply the results of \cref{ss:codecomp} to compute a $\Gamma$-species $\ct{k}$ with respect to some hypothetical action of a group $\Gamma$ whose orbits correspond to generic $k$-trees.
Instead, we will use the additional structure on \emph{rooted} coherently-oriented $k$-trees; with rooting, the cyclic orders around black vertices are converted into linear orders, for which there is a natural action of $\symgp{k+1}$.

\subsection{Group actions on $k$-coding trees}\label{ss:actct}

We have noted previously that every labeled $k$-tree admits exactly $k!$ coherent orientations.
Thus, there are $k!$ distinct $k$-coding trees associated to each labeled $k$-tree, which differ only in the $\specname{C}_{k+1}$-structures on their black vertices.
Consider a rooted $k$-coding tree $T$ and a black vertex $v$ which is not the root vertex.
Then one white neighbor of $v$ is the `parent' of $v$ (in the sense that it lies on the path from $v$ to the root).
We thus can convert the cyclic order on the $k+1$ white neighbors of $v$ to a linear order by choosing the parent white neighbor to be last.
There is a natural, transitive, label-independent action of $\symgp{k+1}$ on the set of such linear orders which induces an action on the cyclic orders from which the linear orders are derived.
However, only elements of $\symgp{k+1}$ which fix $k+1$ will respect the structure around the black vertex we have chosen, since its parent white vertex must remain last.

In addition, if we simply apply the action of some $\sigma \in \symgp{k+1}$ to the order on white neighbors of $v$, we change the coherently-oriented $k$-tree $\beta^{-1} \pbrac{T}$ to which $T$ is associated in such a way that it no longer corresponds to the same unoriented $k$-tree.
Let $t$ denote the unoriented $k$-tree associated to $\beta^{-1} \pbrac{T}$; then there exists a coherent orientation of $t$ which agrees with orientation around $v$ induced by $\sigma$.
The $k$-coding tree $T'$ corresponding to this new coherent orientation has the same underlying bicolored tree as $T$ but possibly different orders around its black vertices.
If we think of the $k$-coding tree $T'$ as the image of $T$ under a global action of $\sigma$, orbits under all of $\symgp{}$ will be precisely the classes of $k$-coding trees corresponding to all coherent orientations of specified $k$-trees, allowing us to study unoriented $k$-trees as quotients.
The orientation of $T'$ will be that obtained by applying $\sigma$ at $v$ and then recursively adjusting the other cyclic orders so that fronts which were mirror are made mirror again.
This will ensure that the combinatorial structure of the underlying $k$-tree $t$ is preserved.

Therefore, when we apply some permutation $\sigma \in \symgp{k+1}$ to the white neighbors of a black vertex $v$, we must also permute the cyclic orders of the descendant black vertices of $v$.
In particular, the permutation $\sigma'$ which must be applied to some immediate black descendant $v'$ of $v$ is precisely the permutation on the linear order of white neighbors of $v'$ induced by passing over the mirror bijection from $v'$ to $v$, applying $\sigma$, and then passing back.
We can express this procedure in formulaic terms:
\begin{theorem}
  \label{thm:rhodef}
  If a permutation $\sigma \in \symgp{k+1}$ is applied to linearized orientation of a black vertex $v$ in rooted $k$-coding tree, the permutation which must be applied to the linearized orientation a child black vertex $v'$ which was attached to the $i$th white child of $v$ (with respect to the linear ordering induced by the orientation) to preserve the mirror relation is $\rho_{i} \pbrac{\sigma}$, where $\rho_{i}$ is the map given by
  \begin{equation}
    \label{eq:rhodef}
    \rho_{i} \pbrac{\sigma}: a \mapsto \sigma \pbrac{i + a} - \sigma \pbrac{i}
  \end{equation}
  in which all sums and differences are reduced to their representatives modulo $k+1$ in $\cbrac{1, 2, \dots, k+1}$.
\end{theorem}
\begin{proof}
  Let $v'$ denote a black vertex which is attached to $v$ by the white vertex $1$, which we suppose to be in position $i$ in the linear order induced by the original orientation of $v$.
  Let $2$ denote the white child of $v'$ which is $a$th in the linear order induced by the original orientation around $v'$.
  It is mirror to the white child $3$ of $v$ which is $\pbrac{i+a}$th in the linear order induced by the original orientation around $v$.
  After the action of $\sigma$ is applied, vertex $3$ is $\sigma \pbrac{i+a}$th in the new linear order around $v$.
  We require that $2$ is still mirror to $3$, so we must move it to position $\sigma \pbrac{i + a} - \sigma \pbrac{i}$ when we create a new linear order around $v'$.
  This completes the proof.
\end{proof}
This procedure is depicted in \cref{fig:rhoapp}.

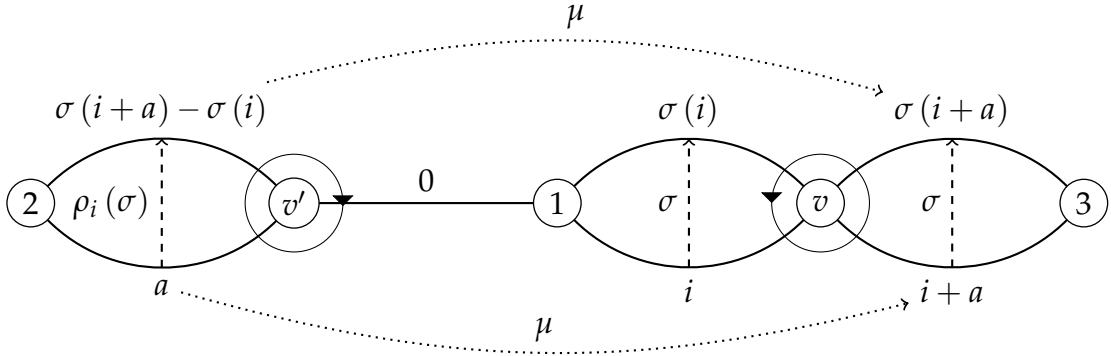
\begin{figure}[htbp]
  \centering
  \begin{tikzpicture}
    \SetGraphUnit{3.5}

    \Vertex[style=xnode,Math=true]{v}
    \cycccwl{(v)};

    \EA[style=ynode](v){3}
    \path (v) edge [bend right=45, thick] node [below](a3d){$i+a$} (3);
    \path (v) edge [bend left=45, thick] node [above](a3u){$\sigma \pbrac{i+a}$} (3);
    \path (a3d) edge [->, dashed, thick] node [auto]{$\sigma$} (a3u);

    \WE[style=ynode](v){1}
    \path (v) edge [bend left=45, thick] node [below](a1d){$i$} (1);
    \path (v) edge [bend right=45, thick] node [above](a1u){$\sigma \pbrac{i}$} (1);
    \path (a1d) edge [->, dashed, thick] node [auto]{$\sigma$} (a1u);
    
    \WE[style=xnode, Math=true](1){v'}
    \cyccwr{(v')};
    \path (v') edge [thick] node [auto](b1){$0$} (1);

    \WE[style=ynode](v'){2}
    \path (v') edge [bend left=45, thick] node [below](b2d){$a$} (2);
    \path (v') edge [bend right=45, thick] node [above](b2u){$\sigma \pbrac{i+a} - \sigma \pbrac{i}$} (2);
    \path (b2d) edge [->, dashed, thick] node [auto]{$\rho_{i} \pbrac{\sigma}$} (b2u);

    \path (b2d) edge [->, dotted, thick, bend right=15] node [auto]{$\mu$} (a3d);
    \path (b2u) edge [->, dotted, thick, bend left=15] node [auto]{$\mu$} (a3u);
  \end{tikzpicture}
  \caption[The permutation-modifying map $\rho$]{Application of a permutation $\sigma$ to the orientation of a non-root black vertex $v$.
    The vertices $2$ and $3$ are mirror in the original orientation (lower set of edges), as shown by the arrows $\mu$, so we must preserve this mirror relation when we apply $\sigma$.
    The permutation $\sigma$ moves $3$ from the $\pbrac{i+a}$th place to the $\sigma \pbrac{i+a}$th, so $\rho_{i} \pbrac{\sigma}$ must carry $2$ from the $a$th place to the $\pbrac{\sigma \pbrac{i+a} - \sigma \pbrac{i}}$th.}
  \label{fig:rhoapp}
\end{figure}

As an aside, we note that, although the construction $\rho$ depends on $k$, the value of $k$ will be fixed in any given context, so we suppress it in the notation.

Any $\sigma$ which is to be applied to a non-root black vertex $v$ must of course fix $k+1$.
We let $\Delta: \symgp{k} \to \symgp{k+1}$ denote the obvious embedding; then the image of $\Delta$ is exactly the set of $\sigma \in \symgp{k+1}$ which fix $k+1$.
We then have an action of $\symgp{k}$ on non-root black vertices induced by $\Delta$.
(Equivalently, we can think of $\symgp{k}$ as the subgroup of $\symgp{k+1}$ of permutations fixing $k+1$, but the explicit notation $\Delta$ will be of use in later formulas.)

In light of \cref{obs:funcdecompct}, we now wish to adapt these ideas into explicit $\symgp{k}$- and $\symgp{k+1}$-actions on $\ctx{k}$, $\cty{k}$, and $\ctxy{k}$ whose orbits correspond to the various coherent orientations of single underlying rooted $k$-trees.
In the case of a $Y$-rooted $k$-coding tree $T$, if we declare that $\sigma \in \symgp{k}$ acts on $T$ by acting directly (as $\Delta \pbrac{\sigma}$) on each of the black vertices immediately adjacent to the root and then applying $\rho$-derived permutations recursively to their descendants, orbits behave as expected.
The same $\symgp{k}$-action serves equally well for edge-rooted $k$-coding trees, where (for purposes of applying the action of some $\sigma$) we can simply ignore the black vertex in the root.

However, if we begin with an $X$-rooted $k$-coding tree, the cyclic ordering of the white neighbors of the root black vertex has no canonical choice of linearization.
If we make an arbitrary choice of one of the $k+1$ available linearizations, and thus convert to an edge-rooted $k$-coding tree, the full $\symgp{k+1}$-action defined previously can be applied directly to the root vertex.
The orbit under this action of some edge-rooted $k$-coding tree $T$ with a choice of linearization at the root then includes all possible linearizations of the root orders of all possible $X$-rooted $k$-coding trees corresponding to the different coherent orientations of a single $k$-coding tree.

\subsection{$k$-trees as quotients}\label{ss:ktquot}
Since these actions are label-independent, we may now treat $\cty{k}$ and $\ctxy{k}$ as $\symgp{k}$-species and $\ctxy{k}$ as an $\symgp{k+1}$-species.
The $\symgp{k}$- and $\symgp{k+1}$-actions on $\ctxy{k}$ are compatible, but we will make explicit reference to $\ctxy{k}$ as an $\symgp{k}$- or $\symgp{k+1}$-species whenever it is important and not completely clear from context which we mean.
As a result of the above results, we can then relate the rooted $\Gamma$-species forms of $\ct{k}$ to the various ordinary species forms of generic rooted $k$-trees in \cref{thm:dissymk}:
\begin{theorem}
  \label{thm:arootquot}
  For the various rooted forms of the ordinary species $\kt{k}$ as in \cref{thm:dissymk} and the various rooted $\Gamma$-species forms of $\ct{k}$ as in \cref{obs:funcdecompct} as $\symgp{k}$- and $\symgp{k+1}$-species, we have
  \begin{subequations}
    \label{eq:arootquot}
    \begin{align}
      \kty{k} &= \faktor{\cty{k}}{\symgp{k}} \label{eq:ayquot} \\
      \ktxy{k} &= \faktor{\ctxy{k}}{\symgp{k}} \label{eq:axyquot} \\
      \ktx{k} &= \faktor{\ctxy{k}}{\symgp{k+1}} \label{eq:axquot}
    \end{align}
  \end{subequations}
  as isomorphisms of ordinary species, where $\ctxy{k}$ is an $\symgp{k}$-species in \cref{eq:axyquot} and an $\symgp{k+1}$-species in \cref{eq:axquot}.
\end{theorem}

As a result, we have explicit characterizations of all the rooted components of the original dissymmetry theorem, \cref{thm:dissymk}.
To compute the cycle indices of these components (and thus the cycle index of $\kt{k}$ itself), we need only compute the cycle indices of the various rooted $\ct{k}$ species, which we will do using a combination of the functional equations in \cref{eq:ctfunc} and explicit consideration of automorphisms.

\section{Automorphisms and cycle indices}\label{s:ktcycind}
\subsection{$k$-coding trees: $\cty{k}$ and $\ctxy{k}$}\label{ss:ctcycind}
\Cref{cor:dissymkreform} of the dissymmetry theorem for $k$-trees has a direct analogue in terms of cycle indices:
\begin{theorem}
  \label{thm:dissymkci}
  For the various forms of the species $\kt{k}$ as in \cref{s:dissymk}, we have
  \begin{equation}
    \label{eq:dissymkci}
    \ci{\kt{k}} = \ci{\ktx{k}} + \ci{\kty{k}} - \ci{\ctxy{k}}.
  \end{equation}
\end{theorem}

Thus, we need to calculate the cycle indices of the three rooted forms of $\kt{k}$.
From \cref{thm:arootquot} and by \cref{thm:qsci} we obtain:
\begin{theorem}
  \label{thm:aquotci}
  For the various forms of the species $\kt{k}$ as in \cref{s:dissymk} and the various $\symgp{k}$-species and $\symgp{k+1}$-species forms of $\ct{k}$ as in \cref{ss:actct}, we have
  \begin{subequations}
    \label{eq:aquotci}
    \begin{align}
      \ci{\kty{k}} &= \qgci{\symgp{k}}{\cty{k}} = \frac{1}{k!} \sum_{\sigma \in \symgp{k}} \gcielt{\symgp{k}}{\cty{k}}{\sigma} \label{eq:ayquotci} \\
      \ci{\ctxy{k}} &= \qgci{\symgp{k}}{\ctxy{k}} = \frac{1}{k!} \sum_{\sigma \in \symgp{k}} \gcielt{\symgp{k}}{\ctxy{k}}{\sigma} \label{eq:axyquotci} \\
      \ci{\ktx{k}} &= \qgci{\symgp{k+1}}{\ctxy{k}} = \frac{1}{\pbrac{k+1}!} \sum_{\sigma \in \symgp{k+1}} \gcielt{\symgp{k+1}}{\ctxy{k}}{\sigma} \label{eq:axquotci}
    \end{align}
  \end{subequations}
\end{theorem}

We thus need only calculate the various $\Gamma$-cycle indices for the $\symgp{k}$-species and $\symgp{k+1}$-species forms of $\cty{k}$ and $\ctxy{k}$ to complete our enumeration of general $k$-trees.

In \cref{obs:funcdecompct}, the functional equations for the ordinary species $\cty{k}$ and $\ctxy{k}$ both include terms of the form $\specname{L}_{k} \circ \cty{k}$.
The plethysm of ordinary species does have a generalization to $\Gamma$-species, as given in \cref{def:gspeccomp}, but it does not correctly describe the manner in which $\symgp{k}$ acts on linear orders of $\cty{k}$-structures in these recursive decompositions.
Recall from \cref{s:quot} that, for two $\Gamma$-species $F$ and $G$, an element $\gamma \in \Gamma$ acts on an $\pbrac{F \circ G}$-structure (colloquially, `an $F$-structure of $G$-structures') by acting on the $F$-structure and on each of the $G$-structures independently.
In our action of $\symgp{k}$, however, the actions of $\sigma$ on the descendant $\cty{k}$-structures are \emph{not} independent---they depend on the position of the structure in the linear ordering around the parent black vertex.
In particular, if $\sigma$ acts on some non-root black vertex, then $\rho_{i} \pbrac{\sigma}$ acts on the white vertex in the $i$th place, where in general $\rho_{i} \pbrac{\sigma} \neq \sigma$.

Thus, we consider automorphisms of these $\symgp{k}$-structures directly.
First, we consider the component species $X \cdot \specname{L}_{k} \pbracs[big]{\cty{k}}$.
\begin{lemma}
  \label{lem:ctyinvar}
  Let $B$ be a structure of the species $X \cdot \specname{L}_{k} \pbracs[big]{\cty{k}}$.
  Let $W_{i}$ be the $\cty{k}$-structure in the $i$th position in the linear order.
  Then some $\sigma \in \symgp{k}$ acts as an automorphism of $B$ if and only if, for each $i \in \sbrac{k+1}$, we have $\Delta^{-1} \pbrac{\rho_{i} \pbrac{\Delta \sigma}} W_{i} \cong W_{\sigma \pbrac{i}}$.
\end{lemma}

\begin{proof}
  Recall that the action of $\sigma \in \symgp{k}$ is in fact the action of $\Delta \sigma \in \symgp{k+1}$.
  The $X$-label on the black root of $B$ is not affected by the application of $\Delta \sigma$, so no conditions on $\sigma$ are necessary to accommodate it.
  However, the $\specname{L}_{k}$-structure on the white children of the root is permuted by $\Delta \sigma$, and we apply to each of the $W_{i}$'s the action of $\Delta^{-1} \pbrac{\rho_{i} \pbrac{\Delta \sigma}}$.
  Thus, $\sigma$ is an automorphism of $B$ if and only if the combination of applying $\Delta \sigma$ to the linear order and $\Delta^{-1} \pbrac{\rho_{i} \pbrac{\Delta \sigma}}$ to each $W_{i}$ is an automorphism.
  Since $\sigma$ `carries' each $W_{i}$ onto $W_{\sigma \pbrac{i}}$, we must have that $\Delta^{-1} \pbrac{\rho_{i} \pbrac{\Delta \sigma}} W_{i} \cong W_{\sigma \pbrac{i}}$, as claimed.
  That this suffices is clear.
\end{proof}

Consider a structure $T$ of the $\symgp{k}$-species $\cty{k}$ and an element $\sigma \in \symgp{k}$.
As discussed in \cref{ss:codecomp}, $T$ is composed of a $Y$-label and a set of $X \cdot \specname{L}_{k} \pbracs[big]{\cty{k}}$-structures.
The permutation $\sigma$ acts trivially on $Y$ and $\specname{E}$ and acts on each of the component $X \cdot \specname{L}_{k} \pbracs[big]{\cty{k}}$-structures independently.
For each of these component structures, by \cref{lem:ctyinvar}, we have that $\sigma$ is an automorphism if and only if $\Delta \sigma$ carries each $\cty{k}$-structure to its $\Delta^{-1} \pbrac{\rho_{i} \pbrac{\Delta \sigma}}$-image.
Thus, when constructing $\sigma$-invariant $X \cdot \specname{L}_{k} \pbracs[big]{\cty{k}}$-structures, we must construct for each cycle of $\sigma$ a $\cty{k}$-structure which is invariant under the application of \emph{all} the permutations $\Delta^{-1} \pbrac{\rho_{i} \pbrac{\Delta \sigma}}$ which will be applied to it along the cycle.
For $c$ the chosen cycle of $\sigma$, this permutation is $\Delta^{-1} \pbrac{\prod_{i \in c} \rho_{i} \pbrac{\Delta \sigma}}$, where the product is taken over any chosen linearization of the cyclic order of the terms in the cycle.
Once a choice of such a $\cty{k}$-structure for each cycle of $\sigma$ is made, we can simply insert the structures into the $\specname{L}_{k}$-structure to build the desired $\sigma$-invariant $X \cdot \specname{L}_{k} \pbracs[big]{\cty{k}}$-structure.
Accordingly:
\begin{theorem}
  \label{thm:ctyfuncci}
  The $\symgp{k}$-cycle index for the species $\cty{k}$ is characterized by the recursive functional equation
  \begin{multline}
    \label{eq:ctyfuncci}
    \gcielt{\symgp{k}}{\cty{k}}{\sigma} = p_{1} \sbrac{y} \\
    \times \ci{\specname{E}} \circ \pbracs[Big]{p_{1} \sbrac{x} \cdot \prod_{c \in C \pbrac{\sigma}} \gci{\symgp{k}}{\cty{k}} \pbracs[Big]{\Delta^{-1} \prod_{i \in c} \rho_{i} \pbrac{\Delta \sigma}} \pbrac{p_{\abs{c}} \sbrac{x}, p_{2 \abs{c}} \sbrac{x}, \dots; p_{\abs{c}} \sbrac{y}, p_{2 \abs{c}} \sbrac{y}, \dots}}.
  \end{multline}
  where $C \pbrac{\sigma}$ denotes the set of cycles of $\sigma$ (as a $k$-permutation) and the inner product is taken with respect to any choice of linearization of the cyclic order of the elements of $c$.
\end{theorem}

The situation for the $\symgp{k+1}$-species $\ctxy{k}$ is almost identical.
Recall from \cref{ss:actct} that $\sigma \in \symgp{k+1}$ acts on a $\ctxy{k}$-structure $T$ by applying $\sigma$ directly to the linear order on the $k+1$ white neighbors of the root black vertex and applying $\rho$-variants of $\sigma$ recursively to their descendants.
Thus, we once again need only require that, along each cycle of $\sigma$, the successive white-vertex structures are pairwise isomorphic under the action of the appropriate $\rho_{i} \pbrac{\sigma}$.
Thus, we again need only choose for each cycle $c \in C \pbrac{\sigma}$ a $\cty{k}$-structure which is invariant under $\prod_{i \in c} \rho_{i} \pbrac{\sigma}$.
Accordingly:
\begin{theorem}
  \label{thm:ctxyfuncci}
  The $\symgp{k+1}$-cycle index for the species $\ctxy{k}$ is given by
  \begin{multline}
    \label{eq:ctxyfuncci}
    \gcielt{\symgp{k+1}}{\ctxy{k}}{\sigma} = p_{1} \sbrac{x} \\
    \times \prod_{c \in C \pbrac{\sigma}} \gci{\symgp{k}}{\cty{k}} \pbracs[Big]{\prod_{i \in c} \rho_{i} \sbrac{\sigma}} \pbrac{p_{\abs{c}} \sbrac{x}, p_{2 \abs{c}} \sbrac{x}, \dots, p_{\abs{c}} \sbrac{y}, p_{2 \abs{c}} \sbrac{y}, \dots}.
  \end{multline}
  under the same conditions as \cref{thm:ctyfuncci}.
\end{theorem}

Terms of the form $\prod_{i \in c} \rho_{i} \pbrac{\sigma}$ appear in \cref{eq:ctyfuncci,eq:ctxyfuncci}.
For the simplification of calculations, we note here a two useful results about these products.

First, we observe that certain $\rho$-maps preserve cycle structure:
\begin{lemma}
  \label{lem:rhofp}
  Let $\sigma \in \symgp{k}$ be a permutation of which $i \in \sbrac{k}$ is a fixed point and let $\lambda$ be the map sending each permutation in $\symgp{k}$ to its cycle type as a partition of $k$.
  Then $\lambda \pbrac{\rho_{i} \pbrac{\sigma}} = \lambda \pbrac{\sigma}$.
\end{lemma}
\begin{proof}
  Suppose $i + a \in \sbrac{k}$ is in an $l$-cycle of $\sigma$.
  Then
  \begin{align*}
    \pbrac{\rho_{i} \pbrac{\sigma}}^{j} \pbrac{a} =& \pbrac{\rho_{i} \pbrac{\sigma}}^{j - 1} \pbrac{\sigma \pbrac{i + a} - \sigma \pbrac{i}} \\
    =& \pbrac{\rho_{i} \pbrac{\sigma}}^{j - 2} \pbrac{\sigma \pbrac{i + \sigma \pbrac{i + a} - \sigma \pbrac{i}} - \sigma \pbrac{i}} \\
    =& \pbrac{\rho_{i} \pbrac{\sigma}}^{j - 2} \pbrac{\sigma^{2} \pbrac{i + a} - \sigma^{2} \pbrac{i}} \\
    &\vdots \\
    =& \sigma^{j} \pbrac{i + a} - \sigma^{j} \pbrac{i}
  \end{align*}
  But the values of $\pbrac{\rho_{i} \pbrac{\sigma}}^{j} \pbrac{a} = \sigma^{j} \pbrac{i + a} - \sigma^{j} \pbrac{i}$ are all distinct for $j \leq l$, since $i + a$ is in an $l$-cycle and $i$ is a fixed point of $\sigma$.
  Furthermore, $\pbrac{\rho_{i} \pbrac{\sigma}}^{l} \pbrac{a} = \sigma^{l} \pbrac{i + a} = i+a$.
  Thus, $a$ is in an $l$-cycle of $\rho_{i} \pbrac{\sigma}$.
  This establishes a length-preserving bijection between cycles of $\rho_{i} \pbrac{\sigma}$ and cycles of $\sigma$, so their cycle types are equal.
\end{proof}

But then we note that the products in the above theorems are in fact permutations obtained by applying such $\rho$-maps:
\begin{lemma}
  \label{lem:rhoprod}
  Let $\sigma \in \symgp{k}$ be a permutation with a cycle $c$.
  Then $\lambda \pbrac{\prod_{i \in c} \rho_{i} \pbrac{\sigma}}$ is determined by $\lambda \pbrac{\sigma}$ and $\abs{c}$.
\end{lemma}
\begin{proof}
  Let $c = \pbrac{c_{1}, c_{2}, \dots, c_{\abs{c}}}$.
  First, we calculate:
  \begin{align*}
    \prod_{i = 1}^{\abs{c}} \rho_{c_{i}} \pbrac{\sigma} =& \rho_{c_{\abs{c}}} \pbrac{\sigma} \circ \dots \circ \rho_{c_{2}} \pbrac{\sigma} \circ \rho_{c_{1}} \pbrac{\sigma} \\
  =& \rho_{c_{\abs{c}}} \pbrac{\sigma} \circ \dots \circ \rho_{c_{2}} \pbrac{\sigma} \pbrac{a \mapsto \sigma \pbrac{c_{1} + a} - \sigma \pbrac{c_{1}}} \\
  =& \rho_{c_{\abs{c}}} \pbrac{\sigma} \circ \dots \circ \rho_{c_{3}} \pbrac{\sigma} \pbrac{a \mapsto \sigma \pbrac{c_{2} + \sigma \pbrac{c_{1} + a} - \sigma \pbrac{c_{1}}} - \sigma \pbrac{c_{2}}} \\
  =& \rho_{c_{\abs{c}}} \pbrac{\sigma} \circ \dots \circ \rho_{c_{3}} \pbrac{\sigma} \pbrac{a \mapsto \sigma^{2} \pbrac{c_{1} + a} - \sigma^{2} \pbrac{c_{1}}} \\
  &\vdots \\
  =& a \mapsto \sigma^{\abs{c}} \pbrac{c_{1} + a} - \sigma^{\abs{c}} \pbrac{c_{1}} \\
  =& \rho_{c_{1}} \pbrac{\sigma^{\abs{c}}}.
  \end{align*}
  But $c_{1}$ is a fixed point of $\sigma^{\abs{c}}$, so by the result of \cref{lem:rhofp}, this has the same cycle structure as $\sigma^{\abs{c}}$, which in turn is determined by $\lambda \pbrac{\sigma}$ and $\abs{c}$ as desired.
\end{proof}

From this and the fact that the terms of $X$-degree $1$ in all $\gci{\cty{k}}{\symgp{k}}$ and $\gci{\ctxy{k}}{\symgp{k+1}}$ are equal (to $p_{1} \sbrac{x} p_{1} \sbrac{y}^{k+1}$), we can conclude that:
\begin{theorem}
  \label{thm:ctciclassfunc}
  $\gcielt{\cty{k}}{\symgp{k}}{\sigma}$ and $\gcielt{\ctxy{k}}{\symgp{k+1}}{\sigma}$ are class functions of $\sigma$ (that is, they are constant over permutations of fixed cycle type).
\end{theorem}

This will simplify computational enumeration of $k$-trees significantly, since the number of partitions of $k$ grows exponentially while the number of permutations of $\sbrac{k}$ grows factorially.

\subsection{$k$-trees: $\kt{k}$}\label{ss:ktcycind}
We now have all the pieces in hand to apply \cref{thm:dissymkci} to compute the cycle index of the species $\kt{k}$ of general $k$-trees.
\Cref{eq:dissymkci} characterizes the cycle index of the generic $k$-tree species $\kt{k}$ in terms of the cycle indices of the rooted species $\ktx{k}$, $\kty{k}$, and $\ctxy{k}$; \cref{thm:arootquot} gives the cycle indices of these three rooted species in terms of the $\Gamma$-cycle indices $\gci{\symgp{k}}{\cty{k}}$, $\gci{\symgp{k}}{\ctxy{k}}$, and $\gci{\symgp{k+1}}{\ctxy{k}}$; and, finally, \cref{thm:ctyfuncci,thm:ctxyfuncci} give these $\Gamma$-cycle indices explicitly.
By tracing the formulas in \cref{eq:ctyfuncci,eq:ctxyfuncci} back through this sequence of functional relationships, we can conclude:
\begin{theorem}[Cycle index for the species of $k$-trees]
  \label{thm:akci}
  For $\mathfrak{a}_{k}$ the species of general $k$-trees, $\gci{\symgp{k}}{\cty{k}}$ as in \cref{eq:ctyfuncci}, and $\gci{\symgp{k+1}}{\ctxy{k}}$ as in \cref{eq:ctxyfuncci} we have:
  \label{thm:ktreecyc}
  \begin{subequations}
    \label{eq:akci}
    \begin{align}
      \ci{\kt{k}} &= \frac{1}{\pbrac{k+1}!} \sum_{\sigma \in \symgp{k+1}} \gcielt{\symgp{k+1}}{\ctxy{k}}{\sigma} + \frac{1}{k!} \sum_{\sigma \in \symgp{k}} \gcielt{\symgp{k}}{\cty{k}}{\sigma} - \frac{1}{k!} \sum_{\sigma \in \symgp{k}} \gcielt{\symgp{k}}{\ctxy{k}}{\sigma} \label{eq:akciexplicit} \\
      &= \qgci{\symgp{k+1}}{\ctxy{k}} + \qgci{\symgp{k}}{\cty{k}} - \qgci{\symgp{k}}{\ctxy{k}}. \label{eq:akciquot}
    \end{align}
  \end{subequations}

\end{theorem}

\Cref{eq:akci} in fact represents a recursive system of functional equations, since the formulas for the $\Gamma$-cycle indices of $\cty{k}$ and $\ctxy{k}$ are recursive.
Computational methods can yield explicit enumerative results.
However, a bit of care will allow us to reduce the computational complexity of this problem significantly.

\section{Unlabeled enumeration and the generating function $\tilde{\mathfrak{a}}_{k} \pbrac{x}$}\label{ss:ktunlenum}
\Cref{eq:akci} in \cref{thm:akci} gives a recursive formula for the cycle index of the ordinary species $\kt{k}$ of $k$-trees.
The number of unlabeled $k$-trees with $n$ hedra is historically an open problem, but by application of \cref{thm:ciogf} the ordinary generating function counting such structures can be extracted from the cycle index $\ci{\kt{k}}$.
Actually computing terms of the cycle index in order to derive the coefficients of the generating function is, however, a computationally expensive process, since the cycle index is by construction a power series in two infinite sets of variables.
The computational process can be simplified significantly by taking advantage of the relatively straightforward combinatorial structure of the structural decomposition used to derive the recursive formulas for the cycle index.

Recall from \cref{thm:gciogf} that, for a $\Gamma$-species $F$, the ordinary generating function $\tilde{F}_{\gamma} \pbrac{x}$ counting unlabeled $\gamma$-invariant $F$-structures is given by
\[\tilde{F} \pbracs[big]{\gamma} \pbrac{x} = \gcieltvars{\Gamma}{F}{\gamma}{x, x^2, x^3, \dots}\]
and that the ordinary generating function for counting unlabeled $\nicefrac{F}{\Gamma}$-structures is given by
\[\tilde{F} \pbrac{x} = \frac{1}{\abs{\Gamma}} \sum_{\gamma \in \Gamma} \tilde{F} \pbracs[big]{\gamma} \pbrac{x}.\]
These formula admits an obvious multisort extension, but we in fact wish to count $k$-trees with respect to just one sort of label (the $X$-labels on hedra), so we will not deal with multisort here.
Each of the two-sort cycle indices in this chapter can be converted to one-sort by substituting $y_{i} = 1$ for all $i$.
For the rest of this section, we will deal directly with these one-sort versions of the cycle indices.

We begin by considering the explicit recursive functional equations in \cref{thm:ctyfuncci,thm:ctxyfuncci}.
In each case, by the above, the ordinary generating function is exactly the result of substituting $p_{i} \sbrac{x} = x^{i}$ into the given formula.
Thus, we have:
\begin{theorem}
  \label{thm:ctrhoogf}
  For $\cty{k}$ the $\symgp{k}$-species of $Y$-rooted $k$-coding trees and $\ctxy{k}$ the $\symgp{k+1}$-species of edge-rooted $k$-coding trees, the corresponding $\Gamma$-ordinary generating functions are given by
  \begin{subequations}
    \label{eq:ctrhoogf}
    \begin{align}
      \widetilde{\cty{k}} \pbrac{\sigma} \pbrac{x} &= \exp \pbracs[Big]{ \sum_{n \geq 1} \frac{x^{n}}{n} \cdot \prod_{c \in C \pbrac{\sigma^{n}}} \widetilde{\cty{k}} \pbracs[Big]{\Delta^{-1} \prod_{i \in c} \rho_{i} \pbracs[big]{\Delta \sigma^{n}}} \pbrac{x^{\abs{c}}}} \label{eq:ctyrhoogf} \\
      \intertext{and}
      \widetilde{\ctxy{k}} \pbrac{\sigma} \pbrac{x} &= x \cdot \prod_{c \in C \pbrac{\sigma}} \widetilde{\cty{k}} \pbracs[Big]{\prod_{i \in c} \rho_{i} \pbrac{\sigma}} \pbracs[big]{x^{\abs{c}}}. \label{eq:ctxyrhoogf}
    \end{align}
  \end{subequations}
  where $\widetilde{\cty{k}}$ is an $\symgp{k}$-generating function and $\widetilde{\ctxy{k}}$ is an $\symgp{k+1}$-generating function.
\end{theorem}

However, as a consequence of \cref{thm:ctciclassfunc}, we can simplify these expressions significantly:
\begin{corollary}
  \label{cor:ctogf}
  For $\cty{k}$ the $\symgp{k}$-species of $Y$-rooted $k$-coding trees and $\ctxy{k}$ the $\symgp{k+1}$-species of edge-rooted $k$-coding trees, the corresponding $\Gamma$-ordinary generating functions are given by
  \begin{subequations}
    \label{eq:ctogf}
    \begin{align}
      \widetilde{\cty{k}} \pbrac{\lambda} \pbrac{x} &= \exp \pbracs[Big]{\sum_{n \geq 1} \frac{x^{n}}{n} \cdot \prod_{i \in \lambda^{n}} \widetilde{\cty{k}} \pbracs[big]{\lambda^{i}} \pbracs[big]{x^{i}}} \label{eq:ctyogf} \\
      \intertext{and}
      \widetilde{\ctxy{k}} \pbrac{\lambda} \pbrac{x} &= x \cdot \prod_{i \in \lambda} \widetilde{\cty{k}} \pbracs[big]{\lambda^{i}} \pbracs[big]{x^{i}} \label{eq:ctxyogf}
    \end{align}
  \end{subequations}
  where $\lambda^{i}$ denotes the $i$th `partition power' of $\lambda$ --- that is, if $\sigma$ is any permutation of cycle type $\lambda$, then $\lambda^{i}$ denotes the cycle type of $\sigma^{i}$ --- and where $f \pbrac{\lambda} \pbrac{x}$ denotes the value of $f \pbrac{\sigma} \pbrac{x}$ for every $\sigma$ of cycle type $\lambda$.
\end{corollary}

As in \cref{thm:ctyfuncci}, we have recursively-defined functional equations, but these are recursions of power series in a single variable, so computing their terms is much less computationally expensive.
Also, as an immediate consequence of \cref{thm:ctciclassfunc}, we have that $\widetilde{\cty{k}}$ and $\widetilde{\ctxy{k}}$ are class functions of $\sigma$, so we can restrict our computational attention to cycle-distinct permutations.

Moreover, the cycle index of the species $\kt{k}$, as seen in \cref{eq:akci}, is given simply in terms of quotients of the cycle indices of the two $\Gamma$-species $\cty{k}$ and $\ctxy{k}$.
Thus, we also have:
\begin{theorem}
  \label{thm:akrhoogf}
  For $\kt{k}$ the species of $k$-trees and $\widetilde{\cty{k}}$ and $\widetilde{\ctxy{k}}$ as in \cref{thm:ctrhoogf}, we have
  \begin{equation}
    \label{eq:akrhoogf}
    \tilde{\mathfrak{a}}_{k} \pbrac{x} = \frac{1}{\pbrac{k+1}!} \sum_{\sigma \in \symgp{k+1}} \widetilde{\ctxy{k}} \pbrac{\sigma} \pbrac{x} + \frac{1}{k!} \sum_{\sigma \in \symgp{k}} \widetilde{\cty{k}} \pbrac{x} \pbrac{\sigma} - \frac{1}{k!} \sum_{\sigma \in \symgp{k}} \widetilde{\ctxy{k}} \pbrac{\sigma} \pbrac{x}.
  \end{equation}
\end{theorem}

Again, as a consequence of \cref{thm:ctciclassfunc} by way of \cref{cor:ctogf}, we can instead write
\begin{corollary}
  For $\kt{k}$ the species of $k$-trees and $\widetilde{\cty{k}}$ and $\widetilde{\ctxy{k}}$ as in \cref{cor:ctogf}, we have
  \begin{equation}
    \label{eq:akogf}
    \tilde{\mathfrak{a}}_{k} \pbrac{x} = \sum_{\lambda \vdash k+1} \frac{1}{z_{\lambda}} \widetilde{\ctxy{k}} \pbrac{\lambda} \pbrac{x} + \sum_{\lambda \vdash k} \frac{1}{z_{\lambda}} \widetilde{\cty{k}} \pbrac{\lambda} \pbrac{x} - \sum_{\lambda \vdash k} \frac{1}{z_{\lambda}} \widetilde{\ctxy{k}} \pbrac{\lambda \cup \cbrac{1}} \pbrac{x}.
  \end{equation}
\end{corollary}

This direct characterization of the ordinary generating function of unlabeled $k$-trees, while still recursive, is much simpler computationally than the characterization of the full cycle index in \cref{eq:akci}.
For computation of the number of unlabeled $k$-trees, it is therefore much preferred.
Classical methods for working with recursively-defined power series suffice to extract the coefficients quickly and efficiently.
The results of some such explicit calculations are presented in \cref{s:ktenum}.

\section{Special-case behavior for small $k$}
Many of the complexities of the preceding analysis apply only for $k$ sufficiently large.
We note here some simplifications that are possible when $k$ is small.

\subsection{Ordinary trees ($k = 1$)}
When $k = 1$, an $\kt{k}$-structure is merely an ordinary tree with $X$-labels on its edges and $Y$-labels on its vertices.
There is no internal symmetry of the form that the actions of $\symgp{k}$ are intended to break.
The actions of $\symgp{2}$ act on ordinary trees rooted at a \emph{directed} edge, with the nontrivial element $\tau \in \symgp{2}$ acting by reversing this orientation.
The resulting decomposition from the dissymmetry theorem in \cref{thm:dissymk} and the recursive functional equations of \cref{obs:funcdecompct} then clearly reduce to the classical dissymmetry analysis of ordinary trees.

\subsection{$2$-trees}
When $k=2$, there is a nontrivial symmetry at fronts (edges); two triangles may be joined at an edge in two distinct ways.
The imposition of a coherent orientation on a $2$-tree by directing one of its edges breaks this symmetry; the action of $\symgp{2}$ by reversal of these orientations gives unoriented $2$-trees as its orbits.
The defined action of $\symgp{3}$ on edge-rooted oriented triangles is simply the classical action of the dihedral group $D_{6}$ on a triangle, and its orbits are unoriented, unrooted triangles.
We further note that $\rho_{i}$ is the trivial map on $\symgp{2}$ and that $\rho_{i} \pbrac{\sigma} = \pbrac{1\ 2}$ for $\sigma \in \symgp{3}$ if and only if $\sigma$ is an odd permutation, both regardless of $i$.
We then have that:
\begin{subequations}
  \label{eq:rest2trees}
  \begin{align}
    \gci{\symgp{2}}{\cty{2}} &= p_{1} \sbrac{y} \cdot \ci{\specname{E}} \circ \pbracs[Big]{p_{1} \sbrac{x} \cdot \prod_{c \in C \pbrac{\sigma}} \gcieltvars{\symgp{2}}{\cty{2}}{e}{p_{\abs{c}} \sbrac{x}, p_{2 \abs{c}} \sbrac{x}, \dots; p_{\abs{c}} \sbrac{y}, p_{2 \abs{c}} \sbrac{y}, \dots}} \label{eq:ctyfuncci2} \\
    \gci{\symgp{3}}{\ctxy{2}} &= p_{1} \sbrac{x} \cdot \prod_{c \in C \pbrac{\sigma}} \gci{\symgp{2}}{\cty{2}} \pbracs[big]{\rho \pbrac{\sigma}^{\abs{c}}} \pbrac{p_{\abs{c}} \sbrac{x}, p_{2 \abs{c}} \sbrac{x}, \dots; p_{\abs{c}} \sbrac{y}, p_{2 \abs{c}} \sbrac{y}, \dots}. \label{eq:ctxyfuncci2}
  \end{align}
\end{subequations}
where, by abuse of notation, we let $\rho$ represent any $\rho_{i}$.
By the previous, the argument $\rho \pbrac{\sigma}^{\abs{c}}$ in \cref{eq:ctxyfuncci2} is $\tau$ if and only if $\sigma$ is an odd permutation and $c$ is of odd length.
This analysis and the resulting formulas for the cycle index $\ci{\kt{2}}$ are essentially equivalent to those derived in \cite{gessel:spec2trees}.

\appendix
\chapter{Computation in species theory}\label{c:comp}
\section{Cycle indices of compositional inverse species}\label{s:compinv}
In \cref{s:nbp}, our results included two references to the compositional inverse $\specname{CBP}^{\bullet \abrac{-1}}$ of the species $\specname{CBP}^{\bullet}$.
Although we have not explored computational methods in depth here, the question of how to compute the cycle index of the compositional inverse of a specified species efficiently is worth some consideration.
Several methods are available, including one developed in \cite[4.2.19]{bll:species} as part of the proof that arbitrary species have compositional inverses, but our preferred method is one of iterated substitution.

Suppose that $\Psi$ is a species (with known cycle index) of the form $X + \Psi_{2} + \Psi_{3} + \dots$ where $\Psi_{i}$ is the restriction of $\Psi$ to structures on sets of cardinality $i$ and that $\Phi$ is the compositional inverse of $\Psi$.
Then $\Psi \circ \Phi = X$ by definition, but by hypothesis
\begin{equation*}
  X = \Psi \circ \Phi = \Phi + \Psi_{2} \pbrac{\Phi} + \Psi_{3} \pbrac{\Phi} + \dots
\end{equation*}
also. Thus
\begin{equation}
  \label{eq:compinv}
  \Phi = X - \Psi_{2} \pbrac{\Phi} - \Psi_{3} \pbrac{\Phi} - \dots.
\end{equation}
This recursive equation is the key to our computational method.
To compute the cycle index of $\Phi$ to degree $2$, we begin with the approximation $\Phi \approx X$ and then substitute it into the first two terms of \cref{eq:compinv}: $\Phi \approx X - \Psi_{2} \pbrac{X}$ and thus $\ci{\Phi} \approx \ci{X} - \ci{\Psi_{2}} \circ \ci{X}$.
All terms of degree up to two in this approximation will be correct.
To compute the cycle index of $\Phi$ to degree $3$, we then take this new approximation $\Phi \approx X - \Psi_{2} \pbrac{X}$ and substitute it into the first three terms of \cref{eq:compinv}.
This process can be iterated as many times as are needed; to determine all terms of degree up to $n$ correctly, we need only iterate $n$ times.
With appropriate optimizations (in particular, truncations), this method can run very quickly on a personal computer to reasonably high degrees; we were able to compute $\ci{\specname{CBP}^{\bullet \abrac{-1}}}$ to degree sixteen in thirteen seconds.

\chapter{Enumerative tables}\label{c:enum}
\section{Bipartite blocks}\label{s:bpenum}
With the tools developed in \cref{c:bpblocks}, we can calculate the cycle indices of the species $\mathcal{NBP}$ of nonseparable bipartite graphs to any finite degree we choose using computational methods.
This result can then be used to enumerate unlabeled bipartite blocks.
We have done so here using Sage 1.7.4 \cite{sage} and code listed in \cref{s:bpbcode}.
The resulting values appear in \cref{tab:bpblocks}.

\begin{table}[htb]
  \centering
  \caption{Enumerative data for unlabeled bipartite blocks with $n$ hedra}
  \label{tab:bpblocks}
  \begin{tabular}{r | r r}
    $n$ & Unlabeled \\\hline
    1 & 1 \\
    2 & 1 \\
    3 & 0 \\
    4 & 1 \\
    5 & 1 \\
    6 & 5 \\
    7 & 8 \\
    8 & 42 \\
    9 & 146 \\
    10 & 956
  \end{tabular}
\end{table}

\section{$k$-trees}\label{s:ktenum}
With the recursive functional equations for cycle indices of \cref{s:ktcycind}, we can calculate the explicit cycle index for the species $\kt{k}$ to any finite degree we choose using computational methods; this cycle index can then be used to enumerate both unlabeled and labeled (at fronts, hedra, or both) $k$-trees up to a specified number $n$ of hedra (or, equivalently, $kn + 1$ fronts).
We have done so here for $k \leq 7$ and $n \leq 30$ using Sage 1.7.4 \cite{sage} using code available in \cref{s:ktcode}.
The resulting values appear in \cref{tab:ktrees}.

We note that both unlabeled and hedron-labeled enumerations of $k$-trees stabilize:
\begin{theorem}
  \label{thm:ktreestab}
  For $k \geq n + 2$, the numbers of unlabeled and hedron-labeled $k$-trees are independent of $k$.
\end{theorem}
\begin{proof}
  We show that the species $\kt{k}$ and $\kt{k+1}$ have contact up to order $k+2$ by explicitly constructing a natural bijection.
  We note that in a $\pbrac{k+1}$-tree with no more than $k+2$ hedra, there will exist at least one vertex which is common to \emph{all} hedra.
  For any $k$-tree with no more than $k+2$ hedra, we can construct a $\pbrac{k+1}$-tree with the same number of hedra by adding a single vertex and connecting it by edges to every existing vertex; we can then pass labels up from the $\pbrac{k+1}$-cliques which are the hedra of the $k$-tree to the $\pbrac{k+2}$-cliques which now sit over them.
  The resulting graph will be a $\pbrac{k+1}$-tree whose $\pbrac{k+1}$-tree hedra are adjacent exactly when the $k$-tree hedra they came from were adjacent.
  Therefore, any two distinct $k$-trees will pass to distinct $\pbrac{k+1}$-trees.
  Similarly, for any $\pbrac{k+1}$-tree with no more than $k+2$ hedra, choose one of the vertices common to all the hedra and remove it, passing the labels of $\pbrac{k+1}$-tree hedra down to the $k$-tree hedra constructed from them; again, adjacency of hedra is preserved.
  This of course creates a $k$-tree, and for distinct $\pbrac{k+1}$-trees the resulting $k$-trees will be distinct.
  Moreover, by symmetry the result is independent of the choice of common vertex, in the case there is more than one.
\end{proof}
However, thus far we have neither determined a direct method for computing these stabilization numbers nor identified a straightforward combinatorial characterization of the structures they represent.

\begin{table}[htb]
  \centering
  \caption{Enumerative data for $k$-trees with $n$ hedra}
  \label{tab:ktrees}

  \subfloat[$k = 1$]{
    \label{tab:1trees}
    \begin{tabular}{r | r}
      $n$ & Unlabeled $1$-trees \\ \hline
      0 & 1 \\
      1 & 1 \\
      2 & 1 \\
      3 & 2 \\
      4 & 3 \\
      5 & 6 \\
      6 & 11 \\
      7 & 23 \\
      8 & 47 \\
      9 & 106 \\
      10 & 235 \\
      11 & 551 \\
      12 & 1301 \\
      13 & 3159 \\
      14 & 7741 \\
      15 & 19320 \\
      16 & 48629 \\
      17 & 123867 \\
      18 & 317955 \\
      19 & 823065 \\
      20 & 2144505 \\
      21 & 5623756 \\
      22 & 14828074 \\
      23 & 39299897 \\
      24 & 104636890 \\
      25 & 279793450 \\
      26 & 751065460 \\
      27 & 2023443032 \\
      28 & 5469566585 \\
      29 & 14830871802 \\
      30 & 40330829030
    \end{tabular}
  }
  \qquad
  \subfloat[$k = 2$]{
    \label{tab:2trees}
    \begin{tabular}{r | r}
      $n$ & Unlabeled $2$-trees \\ \hline
      0 & 1 \\
      1 & 1 \\
      2 & 1 \\
      3 & 2 \\
      4 & 5 \\
      5 & 12 \\
      6 & 39 \\
      7 & 136 \\
      8 & 529 \\
      9 & 2171 \\
      10 & 9368 \\
      11 & 41534 \\
      12 & 188942 \\
      13 & 874906 \\
      14 & 4115060 \\
      15 & 19602156 \\
      16 & 94419351 \\
      17 & 459183768 \\
      18 & 2252217207 \\
      19 & 11130545494 \\
      20 & 55382155396 \\
      21 & 277255622646 \\
      22 & 1395731021610 \\
      23 & 7061871805974 \\
      24 & 35896206800034 \\
      25 & 183241761631584 \\
      26 & 939081790240231 \\
      27 & 4830116366008952 \\
      28 & 24927175920361855 \\
      29 & 129047003236769110 \\
      30 & 670024248072778235
    \end{tabular}
  }
\end{table}

\begin{table}[htb]
  \centering
  \ContinuedFloat
  \caption*{Enumerative data for $k$-trees with $n$ hedra, continued}
  \subfloat[$k = 3$]{
    \label{tab:3trees}
    \begin{tabular}{r | r}
      $n$ & Unlabeled $3$-trees \\ \hline
      0 & 1 \\
      1 & 1 \\
      2 & 1 \\
      3 & 2 \\
      4 & 5 \\
      5 & 15 \\
      6 & 58 \\
      7 & 275 \\
      8 & 1505 \\
      9 & 9003 \\
      10 & 56931 \\
      11 & 372973 \\
      12 & 2506312 \\
      13 & 17165954 \\
      14 & 119398333 \\
      15 & 841244274 \\
      16 & 5993093551 \\
      17 & 43109340222 \\
      18 & 312747109787 \\
      19 & 2286190318744 \\
      20 & 16826338257708 \\
      21 & 124605344758149 \\
      22 & 927910207739261 \\
      23 & 6945172081954449 \\
      24 & 52225283886702922 \\
      25 & 394398440097305861 \\
      26 & 2990207055800156659 \\
      27 & 22753619938517594709 \\
      28 & 173727411594289881739 \\
      29 & 1330614569159767263501 \\
      30 & 10221394007530945428347
    \end{tabular}
  }
  \qquad
  \subfloat[$k = 4$]{
    \label{tab:4trees}
    \begin{tabular}{r | r}
      $n$ & Unlabeled $4$-trees \\ \hline
      0 & 1 \\
      1 & 1 \\
      2 & 1 \\
      3 & 2 \\
      4 & 5 \\
      5 & 15 \\
      6 & 64 \\
      7 & 331 \\
      8 & 2150 \\
      9 & 15817 \\
      10 & 127194 \\
      11 & 1077639 \\
      12 & 9466983 \\
      13 & 85252938 \\
      14 & 782238933 \\
      15 & 7283470324 \\
      16 & 68639621442 \\
      17 & 653492361220 \\
      18 & 6276834750665 \\
      19 & 60759388837299 \\
      20 & 592227182125701 \\
      21 & 5808446697002391 \\
      22 & 57289008242377068 \\
      23 & 567939935463185078 \\
      24 & 5656700148512008902 \\
      25 & 56583199285317631541 \\
      26 & 568236762643725657852 \\
      27 & 5727423267612393252616 \\
      28 & 57924486783495226147615 \\
      29 & 587672090447840337304025 \\
      30 & 5979782184127687211698807
    \end{tabular}
  }
\end{table}

\begin{table}[htb]
  \centering
  \ContinuedFloat
  \caption*{Enumerative data for $k$-trees with $n$ hedra, continued}
  \subfloat[$k = 5$]{
    \label{tab:5trees}
    \begin{tabular}{r | r}
      $n$ & Unlabeled $5$-trees \\ \hline
      0 & 1 \\
      1 & 1 \\
      2 & 1 \\
      3 & 2 \\
      4 & 5 \\
      5 & 15 \\
      6 & 64 \\
      7 & 342 \\
      8 & 2321 \\
      9 & 18578 \\
      10 & 168287 \\
      11 & 1656209 \\
      12 & 17288336 \\
      13 & 188006362 \\
      14 & 2105867058 \\
      15 & 24108331027 \\
      16 & 280638347609 \\
      17 & 3310098377912 \\
      18 & 39462525169310 \\
      19 & 474697793413215 \\
      20 & 5754095507495584 \\
      21 & 70216415130786725 \\
      22 & 861924378411516159 \\
      23 & 10636562125193377459 \\
      24 & 131890971196221692874 \\
      25 & 1642577274341274449247 \\
      26 & 20538830517384955820622 \\
      27 & 257767439475728146293796 \\
      28 & 3246108646710813383678978 \\
      29 & 41008581189552637540038747 \\
      30 & 519599497193547405843864376
    \end{tabular}
  }
  \quad
  \subfloat[$k = 6$]{
    \label{tab:6trees}
    \begin{tabular}{r | r}
      $n$ & Unlabeled $6$-trees \\ \hline
      0 & 1 \\
      1 & 1 \\
      2 & 1 \\
      3 & 2 \\
      4 & 5 \\
      5 & 15 \\
      6 & 64 \\
      7 & 342 \\
      8 & 2344 \\
      9 & 19090 \\
      10 & 179562 \\
      11 & 1878277 \\
      12 & 21365403 \\
      13 & 258965451 \\
      14 & 3294561195 \\
      15 & 43472906719 \\
      16 & 589744428065 \\
      17 & 8171396893523 \\
      18 & 115094557122380 \\
      19 & 1642269376265063 \\
      20 & 23679803216530017 \\
      21 & 344396036645439675 \\
      22 & 5045351124912000756 \\
      23 & 74375422235109338507 \\
      24 & 1102368908826371717478 \\
      25 & 16417712341047912048640 \\
      26 & 245566461812077209025580 \\
      27 & 3687384661929075391318298 \\
      28 & 55566472746158319169779382 \\
      29 & 840092106663809502446963972 \\
      30 & 12739517442131428048314937036
    \end{tabular}
  }
\end{table}

\begin{table}[htb]
  \centering
  \ContinuedFloat
  \caption*{Enumerative data for $k$-trees with $n$ hedra, continued}
  \subfloat[$k = 7$]{
    \label{tab:7trees}
    \begin{tabular}{r | r}
      $n$ & Unlabeled $7$-trees \\ \hline
      0 & 1 \\
      1 & 1 \\
      2 & 1 \\
      3 & 2 \\
      4 & 5 \\
      5 & 15 \\
      6 & 64 \\
      7 & 342 \\
      8 & 2344 \\
      9 & 19137 \\
      10 & 181098 \\
      11 & 1922215 \\
      12 & 22472875 \\
      13 & 284556458 \\
      14 & 3849828695 \\
      15 & 54974808527 \\
      16 & 819865209740 \\
      17 & 12655913153775 \\
      18 & 200748351368185 \\
      19 & 3253193955012557 \\
      20 & 53619437319817482 \\
      21 & 895778170144927928 \\
      22 & 15129118461773051724 \\
      23 & 257812223121779545108 \\
      24 & 4426056869082751747930 \\
      25 & 76463433541541506345648 \\
      26 & 1328088941166844504424628 \\
      27 & 23175796698013212039339479 \\
      28 & 406103563562864890670029228 \\
      29 & 7142350290468621849814034057 \\
      30 & 126034923903699365819345698783
    \end{tabular}   
  }
\end{table}

\begin{table}[htb]
  \centering
  \ContinuedFloat
  \caption*{Enumerative data for $k$-trees with $n$ hedra, continued}
  \subfloat[$k = 8$]{
    \label{tab:8trees}
    \begin{tabular}{r | r}
      $n$ & Unlabeled $8$-trees \\ \hline
      0 & 1 \\
      1 & 1 \\
      2 & 1 \\
      3 & 2 \\
      4 & 5 \\
      5 & 15 \\
      6 & 64 \\
      7 & 342 \\
      8 & 2344 \\
      9 & 19137 \\
      10 & 181204 \\
      11 & 1926782 \\
      12 & 22638677 \\
      13 & 289742922 \\
      14 & 3996857019 \\
      15 & 58854922207 \\
      16 & 916955507587 \\
      17 & 14988769972628 \\
      18 & 255067524402905 \\
      19 & 4487202163529135 \\
      20 & 81112295567987808 \\
      21 & 1498874117898285574 \\
      22 & 28195965395340358096 \\
      23 & 538126404726276758908 \\
      24 & 10391826059632904271057 \\
      25 & 202624626664206041379718 \\
      26 & 3982593421723767068438772 \\
      27 & 78804180647706388187446055 \\
      28 & 1568191570016583843925943321 \\
      29 & 31359266621157738864915907470 \\
      30 & 629755261439815181073415721542
    \end{tabular}
  }
\end{table}

\begin{table}[htb]
  \centering
  \ContinuedFloat
  \caption*{Enumerative data for $k$-trees with $n$ hedra, continued}
  \subfloat[$k = 9$]{
    \label{tab:9trees}
    \begin{tabular}{r | r}
      $n$ & Unlabeled $9$-trees \\ \hline
      0 & 1 \\
      1 & 1 \\
      2 & 1 \\
      3 & 2 \\
      4 & 5 \\
      5 & 15 \\
      6 & 64 \\
      7 & 342 \\
      8 & 2344 \\
      9 & 19137 \\
      10 & 181204 \\
      11 & 1927017 \\
      12 & 22652254 \\
      13 & 290351000 \\
      14 & 4019973352 \\
      15 & 59642496465 \\
      16 & 941751344429 \\
      17 & 15724551551655 \\
      18 & 275926445572426 \\
      19 & 5057692869843759 \\
      20 & 96275031338911591 \\
      21 & 1892687812366295682 \\
      22 & 38234411627616084843 \\
      23 & 790120238796588845615 \\
      24 & 16638524087850961727575 \\
      25 & 355878246778832856290372 \\
      26 & 7710423952280397990026132 \\
      27 & 168843592748278228259801752 \\
      28 & 3730285520855433827693340329 \\
      29 & 83027821492843727307516904184 \\
      30 & 1859625249087075723295908757282
    \end{tabular}
  }
\end{table}

\begin{table}[htb]
  \centering
  \ContinuedFloat
  \caption*{Enumerative data for $k$-trees with $n$ hedra, continued}
  \subfloat[$k = 10$]{
    \label{tab:10trees}
    \begin{tabular}{r | r}
      $n$ & Unlabeled $10$-trees \\ \hline
      0 & 1 \\
      1 & 1 \\
      2 & 1 \\
      3 & 2 \\
      4 & 5 \\
      5 & 15 \\
      6 & 64 \\
      7 & 342 \\
      8 & 2344 \\
      9 & 19137 \\
      10 & 181204 \\
      11 & 1927017 \\
      12 & 22652805 \\
      13 & 290391147 \\
      14 & 4022154893 \\
      15 & 59741455314 \\
      16 & 945737514583 \\
      17 & 15871943695637 \\
      18 & 281035862707569 \\
      19 & 5226147900656616 \\
      20 & 101612006684523937 \\
      21 & 2056425123910104429 \\
      22 & 43127730369661586804 \\
      23 & 933229734601789336024 \\
      24 & 20749443766669472108394 \\
      25 & 472211306357077710523863 \\
      26 & 10961384502758318928846970 \\
      27 & 258737420965101611169934566 \\
      28 & 6193917223279376307682721853 \\
      29 & 150039339181032274342778699887 \\
      30 & 3670778410024403632885217999313
    \end{tabular}
  }
\end{table}

\chapter{Code listing}\label{c:code}
Our results in \cref{c:bpblocks,c:ktrees} provide a framework for enumerating bipartite blocks and general $k$-trees.
However, there is significant work to be done adapting the theory into practical algorithms for computing the actual numbers of such structures.
Using the computer algebra system Sage 1.7.4 \cite{sage}, we have done exactly this.
In each case, the script listed may be run with Sage by invoking
\begin{verbatim}
> sage --python scriptname.py args
\end{verbatim}
on a computer with a functioning Sage installation.
Alternatively, each code snippet may be executed in the Sage `notebook' interface starting at the comment ``\texttt{MATH BEGINS HERE}''; in this case, the final \texttt{print\dots} invocation should be replaced with one specifying the desired parameters.

\section{Bipartite blocks}\label{s:bpbcode}
The functional \cref{eq:nbpexp} characterizes the cycle index of the species $\specname{NBP}$ of bipartite blocks.
Python/Sage code to compute the coefficients of the ordinary generating function $\widetilde{\specname{NBP}} \pbrac{x}$ of unlabeled bipartite blocks explicitly follows in \cref{lst:bpcode}.
Specifically, the generating function may be computed to degree $n$ by invoking
\begin{verbatim}
> sage --python bpblocks.py n
\end{verbatim}
on a computer with a functioning Sage installation.

\begin{lstlisting}[caption=Sage code to compute numbers of bipartite blocks (\texttt{bpblocks.py}), label=lst:bpcode, language=Python]
from sage.all import *
import sys

# Take arguments from the command line
args = sys.argv

assert len(args) == 2

nval = int(args[1])

##MATH BEGINS HERE
# Set up a ring of symmetric functions
p = SymmetricFunctionAlgebra(QQ, basis="power")
x = PowerSeriesRing(QQ, 'x').gen()

# Define basic objects
Zx = p[1]

Z2 = SymmetricGroup(2)
e = Z2.identity()
t = Z2.gen() #The non-identity element

# Utility function for composing a generating function into a cycle index
def gf_pleth(f, gf):
    partmapper = lambda part: prod(gf.subs({x:x^i}) for i in part)
    return p._apply_module_morphism(f, partmapper)

# Utility function for computing the ogf of unlabeled structures of a species with cycle index f
def unlabeled(f):
    partmapper = lambda part: x^(part.size())
    return p._apply_module_morphism(f, partmapper)

# Utility function for computing the egf of labeled structures of a species with cycle index f
def labeled(f):
    def partmapper(part):
        if part.length() == part.size():
            return x^(part.length())
        else:
            return 0
    return p._apply_module_morphism(f, partmapper)

# Compute a plethystic inverse of the symmetric function f to degree n
def plethystic_inverse(f, n):
        parent = f.parent()

        fstripped = f - f.restrict_degree(1, exact=True)

        improver = lambda F, i: Zx - fstripped.restrict_degree(i, exact=False).plethysm(F).restrict_degree(i, exact=False)

        finv = Zx

        for i in xrange(2, n+1):
            finv = improver(finv, i)

        return finv

# Utility function to compute the pointing of a cycle index
def pointed(f):
    return Zx*f.derivative_with_respect_to_p1()

# Compute the cycle index of the species E of sets
def Ze(n):
    return balanced_sum(1/part.aut() * prod(p[l] for l in part) for i in xrange(1,n+1) for part in Partitions(i))

# Compute the cycle index of the species Con which is the plethystic inverse of E
def Zcon(n):
    return plethystic_inverse(Ze(n), n)

# Utility function to compute the union of partitions
def union( mu, nu):
    return Partition(sorted(mu.to_list() + nu.to_list(), reverse=true))

# Utility function to compute the partwise multiple of a partition
def partmult( mu, n ):
    return Partition([part * n for part in mu.to_list()])

# Compute the number of bicolored graphs invariant under a permutation of cycle type mu acting on white vertices and a permutation of cycle type nu acting on black vertices
def efixedbcgraphs( mu, nu ):
    return 2^(sum([gcd(i, j) for i in mu for j in nu]))

# Compute the number of bicolored graphs for which a permutation of cycle type mu is a color-reversing isomorphism
def tfixedbcgraphs( mu ):
    return 2^(len(mu) + sum([integer_ceil(p/2) for p in mu]) + sum([gcd(mu[i], mu[j]) for i in range(0, len(mu)) for j in range(i+1, len(mu))]))

# Compute the Z2-cycle index of the species BC of bicolored graphs
def Zbc( z2elt, n ):
    assert z2elt in Z2
    if z2elt == Z2.identity():
        return add([efixedbcgraphs(pair[0], pair[1]) * p(union(pair[0],pair[1])) / (pair[0].aut() * pair[1].aut()) for i in range(1, n+1) for pair in PartitionTuples(i, 2).list()])
    else:
        return add([tfixedbcgraphs( mu ) * p(partmult(mu, 2))/partmult(mu, 2).aut() for i in range(1, integer_floor(n/2)+1) for mu in Partitions(i).list() ])

# Compute the Z2-cycle index of the species CBC of connected bicolored graphs
def Zcbc( z2elt, n ):
    assert z2elt in Z2
    if z2elt == Z2.identity():
        return Zcon(n).plethysm(Zbc(e, n)).restrict_degree(n, exact=False)
    else:
        scale_part = lambda n: lambda m: m.__class__([i*n for i in m])
        pn_pleth = lambda f, n: f.map_support(scale_part(n))
        f = lambda part: prod(pn_pleth(Zbc(t^i, n), i) for i in part)
        return p._apply_module_morphism(Zcon(n), f).restrict_degree(n, exact=False)

# Compute the cycle index of the species CBP of connected bipartite graphs
def Zcbp( n ):
    return 1/2 * (Zcbc(e, n) + Zcbc(t, n))

# Compute the ordinary generating function for nonseparable bipartite graphs
def Znbp_ogf( n ):
    Zcbp_ci = Zcbp(n)
    Zcbp_dotinv_ogf = unlabeled(plethystic_inverse(pointed(Zcbp_ci), n))
    Zcbp_comp_ogf = gf_pleth(Zcbp_ci, Zcbp_dotinv_ogf).add_bigoh(n+1)

    Znbp_prime_ogf = gf_pleth(Zcon(n), (x/Zcbp_dotinv_ogf - 1).add_bigoh(n+1)).add_bigoh(n) + 1

    return Zcbp_comp_ogf + x * Znbp_prime_ogf - x

# Print the result
print Znbp_ogf(nval)
\end{lstlisting}

\section{$k$-trees}\label{s:ktcode}
The recursive functional equations in \cref{eq:ctyogf,eq:ctxyogf,eq:akogf} characterize the ordinary generating function $\tilde{\mathfrak{a}}_{k} \pbrac{x}$ for unlabeled general $k$-trees.
Python/Sage code to compute the coefficients of this generating function explicitly follows in \cref{lst:ktcode}.
Specifically, the generating function for unlabeled $k$-trees may be computed to degree $n$ by invoking
\begin{verbatim}
> sage --python ktrees.py k n
\end{verbatim}
on a computer with a functioning Sage installation.

This code uses the class-function optimization of \cref{thm:ctciclassfunc} extensively; as a result, it is able to compute the number of $k$-trees on up to $n$ hedra quickly even for relatively large $k$ and $n$.
For example, the first thirty terms of the generating function for $8$-trees in \cref{tab:8trees} were computed on a modern desktop-class computer in approximately two minutes.

\begin{lstlisting}[caption=Sage code to compute numbers of $k$-trees (\texttt{ktrees.py}), label=lst:ktcode, language=Python]
from sage.all import *
import sys

# Take arguments from the command line
args = sys.argv

assert len(args) == 3

kval = int(args[1])
nval = int(args[2])

##MATH BEGINS HERE
# Set up a ring of formal power series
psr = PowerSeriesRing(QQ, 'x')
x = psr.gen()

# Compute the generating function for unlabeled Y-rooted k-trees fixed by permutations of a given cycle type mu.
# Note that mu should partition k
@cached_function
def unlY(mu, n):
    if n <= 0:
        return psr(1)
    else:
        ystretcher = lambda c, part: unlY(Partition(partition_power(part, c)), floor((n-1)/c)).subs({x:x**c})
        descendant_pseries = lambda part: prod(ystretcher(c, part) for c in part)
        return sum(x**i/i * descendant_pseries(partition_power(mu, i)).subs({x:x**i}) for i in xrange(1, n+1)).exp(n+1)

# Compute the generating function for unlabeled XY-rooted k-trees fixed by permutations of a given cycle type mu.
# Note that mu should partition k+1
@cached_function
def unlXY(mu, n):
    if n <= 0:
        return psr(0)
    else:
        ystretcher = lambda c: unlY(Partition(partition_power(mu, c)[:-1]), floor((n-1)/c)).subs({x:x**c})
        return (x * prod(ystretcher(c) for c in mu)).add_bigoh(n+1)

# Compute the generating functions for unlabeled X-, Y-, and XY-rooted k-trees using quotients
ax = lambda k, n: sum(1/mu.aut() * unlXY(mu, n) for mu in Partitions(k+1))
ay = lambda k, n: sum(1/mu.aut() * unlY(mu, n) for mu in Partitions(k))
axy = lambda k, n: sum(1/mu.aut() * unlXY(Partition(mu + [1]), n) for mu in Partitions(k))

# Compute the generating function for unlabeled un-rooted k-trees using the dissymmetry theorem
a = lambda k, n: ax(k, n) + ay(k, n) - axy(k, n)

# Print the result
print a(kval, nval)
\end{lstlisting}

\backmatter

\bibliographystyle{amsplain} %bibliography
\bibliography{sources}

\end{document}